\documentclass[10pt, final, journal, letterpaper, oneside, twocolumn]{IEEEtran}

\makeatletter

\let\proof\@undefined
\let\endproof\@undefined

\let\theorem\@undefined
\let\endtheorem\@undefined
\makeatother

\usepackage{epsfig}
\usepackage{amssymb} 
\usepackage[tbtags]{amsmath} 
\usepackage{graphics,eepic,epic}
\usepackage{latexsym}
\usepackage{euscript}
\usepackage{subfigure}
\usepackage{graphics,eepic,epic,psfrag}
\usepackage{color}
\usepackage{tikz} 
\usetikzlibrary{positioning} 
\usepackage{multirow} 
\usepackage{pgfplots} 
\usepackage{footnote} 
\makesavenoteenv{tabular} 
\usepackage{cite} 
\usepackage[all]{xy} 

\usepackage{amsthm}
\newtheorem{thm}{Theorem}
\newtheorem{lemma}{Lemma}

\newtheorem{defn}{Definition}

\newtheorem{remark}{Remark}
\newtheorem{asmp}{\textbf{Assumption}}
\newcommand{\reals}{\mathbb{R}}

\newcommand{\norm}[1]{\left \lVert#1\right \rVert}

\tikzstyle{block} = [draw, fill=white, rectangle, 
    minimum height=2em, minimum width=2em]
\tikzstyle{sum} = [draw, fill=white, circle, node distance=0.5cm, inner sep=0pt, minimum size=0.25cm]
\tikzstyle{input} = [coordinate]
\tikzstyle{output} = [coordinate]
\tikzstyle{pinstyle} = [pin edge={to-,thin,black}]
\tikzstyle{branch}=[fill,shape=circle,minimum size=3pt,inner sep=0pt]

\IEEEoverridecommandlockouts                              

\title{\Huge  A Passivity-Based Approach to Nash Equilibrium Seeking over Networks
}

\author{Dian Gadjov and Lacra Pavel
\thanks{This work was supported by an NSERC Discovery Grant.}%
\thanks{D. Gadjov and L. Pavel are with the Department of Electrical and Computer Engineering, University of Toronto, Toronto, ON, M5S 3G4, Canada %
        {\tt\small dian.gadjov@mail.utoronto.ca}, 
        {\tt\small pavel@control.utoronto.ca}}%
}

\begin{document}

\pdfminorversion=4

\maketitle
\thispagestyle{empty}
\pagestyle{empty}


\begin{abstract}

In this paper we consider the problem of distributed Nash equilibrium (NE) seeking  over networks,  a setting in which players have limited local information. We start from a continuous-time gradient-play dynamics that converges to an NE under strict monotonicity  of the pseudo-gradient and assumes perfect information,  i.e., instantaneous all-to-all player communication. We consider how to modify this gradient-play dynamics in the case of partial, or networked information between players.  We propose an augmented   gradient-play dynamics with correction in which players communicate locally only with their neighbours  to compute an estimate of the other players' actions. We derive the new dynamics based on the reformulation as a multi-agent coordination problem over an undirected graph. We exploit incremental passivity properties and show that a synchronizing, distributed Laplacian feedback can be designed using relative estimates of the neighbours. Under  a strict monotonicity  property of the pseudo-gradient, we show that the augmented gradient-play dynamics converges to consensus on the NE of the game. We further discuss two cases that highlight the tradeoff between properties of the game and the communication graph.
\end{abstract}

\section{Introduction}

We consider  distributed Nash equilibrium (NE) seeking over  networks, where players  have limited local information, over a communication network. This is a research topic of recent interest,   \cite{FKB12}, \cite{Johansson2016},  \cite{stankovic2012distributed},  due to many networked scenarios in which such problems arise such as in wireless communication \cite{LH10}, \cite{CH12}, \cite{AB04}, optical networks \cite{Pavelbook2012}, \cite{Pavel06}, \cite{Pavel09}, distributed constrained convex optimization   \cite{WE11}, \cite{Marden2013}, noncooperative flow control problems \cite{YSM11}, \cite{Alpcan2005}, etc.

We propose a new \emph{continuous-time dynamics} for a general class of  N-player games and prove its convergence to  NE  over a \emph{connected} graph. Our scheme is derived based on reformulating the problem as a multi-agent coordination problem between the players and leveraging passivity properties. Specifically, we endow each player (agent) with an auxiliary state variable that provides an estimate of all other players' actions. For each agent we combine its own gradient-type  dynamics with an integrator-type auxiliary dynamics, driven by some control signal. We design the control signal for each individual player, based on the relative output feedback from its neighbours,  such that these auxiliary state variables agree one with another. The resulting player's dynamics has two components: the \emph{action} component composed of a gradient term (enforcing the move towards minimizing its own cost) and a consensus term, and the \emph{estimate} component  with a consensus term. We call this new dynamics an augmented gradient-play dynamics with correction and estimation.  We prove  it converges to consensus on the NE of the game, under  a monotonicity  property of the extended pseudo-gradient.



\emph{Literature review.} Our work is related to the literature of NE seeking in games  \emph{over networks}. Existing results are almost exclusively developed in discrete-time. The problem of NE seeking under networked communication is considered  in \cite{NS12},\cite{G16}, specifically for the special class of \emph{aggregative} games, where each agent's cost function is coupled to other players' actions through a single, aggregative variable. In \cite{FPAuto2016},  this approach is generalized to a larger class of coupled games: a gossip-based discrete-time algorithm is proposed and convergence was shown for \emph{diminishing step-sizes}. Very recently, discrete-time ADMM-type algorithms with constant step-sizes have been proposed and convergence proved  under \emph{co-coercivity}  of the extended pseudo-gradient,  \cite{FredIFAC2017}, \cite{WeiShiACC2017}.   

In continuous-time, gradient-based dynamics for NE computation have been used  since the work Arrow, \cite{ArrowHurwitzUzawa1951}, \cite{ArrowHurwitzUzawa1958}, \cite{Flam02}, \cite{SA05}. Over networks, gradient-based algorithms are designed in \cite{Marden2013}   based on information from only a set of local neighbouring agents for  games with local utility functions (proved to be state-based potential games). Continuous-time distributed NE seeking dynamics are proposed for a two-network zero-sum game in \cite{G13}. Key assumptions are the additive decomposition of the common objective function as well as other structural assumptions. Based on the max-min formulation, the dynamics takes the form of a saddle-point dynamics, \cite{ArrowHurwitzUzawa1951},  distributed over the agents of each of the two networks, inspired by the optimization framework of \cite{WE11}. 

In this paper we consider a general class of  N-player games,  where  players have limited local information about the others' actions over a communication network. Our work is also related to the distributed optimization framework in \cite{WE11}. However there are several differences between \cite{WE11} or  \cite{G13} and our work. Beside the summable structure  of the common cost function, in \cite{WE11} a critical structural  assumption is the fact that each agent optimizes its cost function over the \emph{full} argument. Then, when an augmented (lifted) space of actions and estimates is considered in the  networked communication case, a lifted cost function is obtained which can be decomposed as a sum of  \emph{separable cost functions, individually convex in their full argument}.  This leads to distributed algorithms, under strict convexity of the individual cost  functions with respect to the \emph{ full} argument. In a strategic game context, the individual convexity  properties with respect to the full argument are too restrictive unless the game is separable to start with. While the game setting has an inherent distributed structure (since each player optimizes its own cost function),  individual (player-by-player) optimization) is over \emph{its own} action. In contrast to distributed optimization, each player's individual action is only \emph{part  of the full action profile} and its cost function is coupled to its opponents' actions, which are under their decision. This key differentiating structural aspect between games and distributed optimization presents technical challenges. 

In this work,  we also consider  an augmented space to deal with the networked communication case, of actions and estimates of others' actions. However, the difficulty is that we do not have an additive decomposition to exploit, and each player only controls/optimizes \emph{a part} of the full argument  on which its own cost function depends on. Our main approach is to highlight and exploit passivity  properties of the  game (pseudo-gradient), gradient-based algorithm/dynamics and network (Laplacian). 
A typical assumption  in games is not individual gradient monotonicity with respect to the full argument, but  rather monotonicity of the pseudo-gradient. The corresponding game assumption we use in the  networked communication case, is monotonicity of  the extended pseudo-gradient, which is equivalent to incremental passivity.

\emph{Contributions. } We consider a general class of N-player games and  develop of a new continuous-time dynamics for  NE seeking dynamics  under networked information. Our approach is based on reformulating the problem as a multi-agent coordination problem and exploiting basic incremental passivity properties of the  pseudo-gradient map. To the best of our knowledge such an approach has not been proposed before. Our contributions are three-fold.
 \emph{First}, we show that under strict monotonicity of the extended pseudo-gradient, the proposed new dynamics  converges over \emph{any connected} graph, unlike \cite{FredIFAC2017}, \cite{WeiShiACC2017}. Our scheme is different from  \cite{FPAuto2016}, due to an extra correction term on the actions' dynamics that arises naturally from the passivity-based design. This term is in fact critical to prove convergence on a single timescale. Essentially, players perform simultaneous consensus of estimates and player-by-player optimization.

 \emph{Secondly}, our passivity-based approach  highlights the tradeoff between properties of the game and those of the communication graph. Under a weaker Lipschitz continuity assumption of the extended pseudo-gradient,  we show that the new dynamics  converges over \emph{any sufficiently connected} graph. Key is the fact that  the Laplacian contribution (or excess passivity) can be used to balance the other terms  that are dependent on the game properties.  \emph{Thirdly},  we relax the connectivity bound on the graph, based on a time-scale separation argument.  This is achieved by modifying  the dynamics of the estimates such that the system approaches quickly the consensus subspace.

The paper is organized as follows.
Section \ref{background} gives the preliminary background.
Section \ref{problem_Statement} formulates the noncooperative game and basic assumptions.
Section \ref{graphCommunication} presents the distributed NE seeking dynamics  and analyzes its equilibrium points.
Section \ref{convergence} analyzes the convergence of the proposed dynamics over a connected graph under various assumptions. 
Section \ref{projection} considers the case of compact action spaces, where projection dynamics are proposed and analyzed.
Numerical examples are given in Section \ref{Simulation} and conclusions in Section \ref{Conclusion}.

\section{Preliminaries} \label{background}
\emph{Notations}. Given a vector $x \in \reals^n$,  $x^T$ denotes its transpose.  Let $x^Ty$ denote the Euclidean inner product of $x,y \in \reals^n$ and $\|x\|$ the Euclidean norm. Let  $A \otimes B$ denote the Kronecker product of matrices $A$ and $B$. The all ones vector is $\mathbf{1}_n = [1, \dotsc, 1]^T \in \reals^n$,  and the all zeros vector is $\mathbf{0}_n = [0, \dotsc, 0]^T \in \reals^n$.  $diag(A_1,\dotsc,A_n)$ denotes the block-diagonal matrix with  $A_i$ on its diagonal. Given a matrix $M\in\reals^{p\times q}$,  $Null(M) = \{x\in\reals^q | Mx = 0\}$ and $Range(M) = \{y\in\reals^p | (\exists x\in\reals^q) \ y = Mx \}$. A function $\Phi:  \reals^n \rightarrow  \reals^n$ is monotone if $(x-y)^T(\Phi(x) - \Phi(y)) \geq 0$, for all $x, y \in \reals^n$, and strictly monotone if the inequality is strict when $x\neq y$.  $\Phi$ is strongly monotone if there exists $\mu >0$ such that $(x-y)^T(\Phi(x) - \Phi(y)) \geq \mu \| x- y \|^2$, for all $x,y \in\reals^n$.  For a differentiable function $V:\reals^n\rightarrow \reals$, $ \nabla V(x) = \frac{\partial V}{\partial x}(x) \in \reals^n$ denotes its gradient.   $V$ is convex, strictly, strongly convex if and only if its gradient $\nabla V$ is monotone, strictly, strongly monotone. 
Monotonicity properties play in variational inequalities the same role as convexity plays in optimization.

\subsection{Projections}


Given a closed, convex set $\Omega \subset \reals^n$, let the interior, boundary and closure of $\Omega$ be denoted by $\text{int} \Omega$, $\partial \Omega$ and $\overline{\Omega}$, respectively. The normal cone of $\Omega$ at a point $x \in {\Omega}$ is defined  as $N_{\Omega}(x) = \{y\in\reals^{n} |  y^T(x'-x) \leq 0, \forall x'\in {\Omega}\}$.  The tangent cone of $\Omega$ at $x \in {\Omega}$  is given as $T_{\Omega}(x) = \overline{\bigcup_{\delta>0} \frac{1}{\delta} (\Omega - x)}$. 
The projection operator of a point $x \in \reals^n$ to the set $\Omega$ is  given by the point  $P_{\Omega}(x) \in \Omega$  such that $\norm{x - P_{\Omega}(x)} \leq \norm{x - x'} $, for all $x'\in \Omega$, or $P_{\Omega}(x) =\text{argmin}_{x' \in \Omega} \norm{x - x'}$. The projection operator of a vector $v \in \reals^n$ at a point $x \in \Omega$ with respect to $\Omega$ is $\Pi_{\Omega}(x, v) = \underset{\delta\to 0_+}{\lim} \frac{P_{\Omega}(x + \delta v) - x}{\delta}$. Note that $\norm{\Pi_{\Omega}(x, v)} \leq \norm{v}$. Given $x \in \partial \Omega$ let $n(x)$ denote the set of  outward unit normals to $\Omega$ at $x$,  $n(x) = \{ y \,  | \, y \in N_{\Omega}(x), \norm{y} =1 \}$. 
By Lemma 2.1 in \cite{NZ96}, if $x\in \text{int} \Omega$, then $\Pi_{\Omega}(x, v) =v$, while if $x \in \partial \Omega$, then 
\begin{align}\label{charactPi}
\Pi_{\Omega}(x, v) =v - \beta(x) \, n^*(x)
\end{align}
where $n^*(x) = \underset{n \in n(x)}{\text{argmax}} \,  v^T \, n $ and $\beta(x) = \max\{ 0,  v^T n^*(x) \}$. Note that if $v \in {T_\Omega(x)}$ for some  $x \in \partial \Omega$, then 
$\underset{n \in n(x)}{\text{sup}} \,  v^T \, n \leq 0$ hence $\beta(x) =0$ and no projection needs to be performed.   The operator  $\Pi_{\Omega}(x, v)$ is equivalent to the projection of the vector $v$ onto the tangent cone $T_{\Omega}(x)$ at $x$, $\Pi_{\Omega}(x, v) = P_{T_\Omega(x)}(v)$. 

A set $C\subseteq \reals^n$ is a cone if  for any $ c\in C$, $\gamma c \in C$ for every $\gamma > 0$. The polar cone of  a convex cone $C$ is given by $C^\circ = \{y\in\reals^n | \, y^Tc \leq 0,\ \forall c\in C\}$.
\begin{lemma}[Moreau's Decomposition Theorem III.3.2.5, \cite{Lemarechal}]\label{lemma:Moreau}
	Let $C \subseteq \reals^n$ and $C^\circ \subseteq \reals^n$ be a closed convex cone and its polar cone, and let $v\in\reals^n$. Then the following are equivalent:
	
	(i) $v_C =  P_{C}(v)$ and  $v_{C^\circ} =  P_{C^\circ}(v)$.
	
	(ii) $v_{C}\in C$, $v_{C^\circ}\in C^\circ$, $v = v_{C} + v_{C^\circ}$,  and $v^T_{C} \, v_{C^\circ} =0$.
\end{lemma}
 Notice that $N_{\Omega}(x)$ is a convex cone and the tangent cone is its polar cone, i.e., $N_{\Omega} (x)= (T_{\Omega}(x))^\circ$, $(N_{\Omega} (x))^\circ= T_{\Omega}(x)$. By Lemma \ref{lemma:Moreau}, for any $x\in \Omega$, any  vector  $v\in \reals^n$ can be decomposed into  tangent  $ v_{T_{\Omega}}\in T_{\Omega}(x)$ and normal components, $v_{N_{\Omega}}\in N_{\Omega}(x)$, 
 \begin{align}\label{v_Moreau}
	v = v_{T_{\Omega}} + v_{N_{\Omega}}
\end{align}
with  $v_{T_{\Omega}} =  P_{T_{\Omega}(x)}(v) =\Pi_{\Omega}(x, v)$, $v_{N_{\Omega}} =  P_{N_{\Omega}(x)}(v)$.

\subsection{Graph theory}

The following are from \cite{AGT01}. An undirected graph $G_c$ is a pair $G_c=(\mathcal{I},E)$ with  $\mathcal{I} = \{1,\dotsc,N\}$ the vertex set and $E\subseteq \mathcal{I}\times \mathcal{I}$ the edge set  such that for $i,j\in \mathcal{I}$, if $(i,j)\in E$, then $(j,i)\in E$. The degree of vertex $i$,  $\text{deg}(i)$, is the number of edges connected to $i$. A path in a graph is a sequence of edges which connects a sequence of vertices. A graph is connected if there is a path between every pair of vertices. In this paper we associate a vertex with a player/agent. An edge 
between agents $i,j \in \mathcal{I} $ exists if agents $i$ and $j$ exchange information. Let $\mathcal{N}_i \subset \mathcal{I}$ denote the set of neighbours of  player $i$. The Laplacian matrix $L\in \reals^{N\times N}$ describes the connectivity of the graph $G_c$, with   $[L]_{ij}  = |\mathcal{N}_i|$, if $i=j$,  $	[L]_{ij}  = -1$, if $j\in\mathcal{N}_i $ and $0$ otherwise. When $G_c$ is an undirected and connected graph, $0$ is a simple eigenvalue of $L$,  $L\mathbf{1}_N = \mathbf{0}_N$,  and all other eigenvalues positive. 
 Let the eigenvalues of $L$ in ascending order be $0<\lambda_2\leq\dotsc\leq\lambda_N$, then $\underset{x\neq 0,\ \mathbf{1}_N^Tx = 0}{min} \, x^TLx = \lambda_2 \|x\|^2_2$, 
$ \underset{x\neq 0}{max} \, x^TLx = \lambda_N \|x\|_2^2$. 
 \subsection{Equilibrium Independent and Incremental Passivity}
The following  are from \cite{MA11},\cite{DN13},\cite{BP15}, \cite{Pavlov08}. Consider $\Sigma$
\begin{align} \label{eq:DynOverall}
\Sigma :
  \begin{cases}
    \dot{x} = f(x,u), & \\
    y = h(x,u), & \\
  \end{cases}
\end{align}
with $x\in \reals^n$,  $u \in \reals^q$ and $y \in \reals^q$, $f$ locally Lipschitz and $h$ continuous. Consider a differentiable  function $V:\reals^n \rightarrow \reals$. The time derivative of $V$ along solutions of (\ref{eq:DynOverall}) is denoted as $\dot{V}(x) = \nabla^T V(x) \, f(x,u)$ or just $\dot{V}$. Let  $\overline{u}$, $\overline{x}$, $\overline{y}$ be  an equilibrium condition, such that $0=f(\overline{x},\overline{u})$, $\overline{y}=h(\overline{x},\overline{u})$. Assume $\exists \overline{U}\subset\reals$ and a continuous function $k_x(\overline{u})$  such that for any constant  $\overline{u} \in \overline{U}$,   $f(k_x(\overline{u}),\overline{u}) = 0$.  The   continuous function  $k_y(\overline{u}) = h(k_x(\overline{u}),\overline{u})$ is the  equilibrium input-output map. 
Equilibrium independent passivity (EIP) requires $\Sigma$ to be passive independent of the equilibrium point. 
\begin{defn}\label{def:EIP}
System $\Sigma$ (\ref{eq:DynOverall}) is Equilibrium Independent Passive (EIP) if it is passive with respect to $\overline{u}$ and $\overline{y}$; that is for every $ \overline{u} \in \overline{U}$ there exists a differentiable, positive semi-definite storage function $V: \reals^n \to \reals$ such that $V(\overline{x}) = 0$ and, for all $u \in \reals^q$, $x \in \reals^n$,
		$$\dot{V}(x)  \leq  (y-\overline{y})^T(u-\overline{u})$$
\end{defn}
A slight refinement to the EIP definition can be made to handle the case where $k_y(\overline{u})$ is not a function but is a map.  
An EIP system  with a map $k_y(\overline{u})$  is called maximal EIP (MEIP) when  $k_y(\overline{u})$ is maximally monotone, e.g. an integrator, \cite{DN13}. The parallel interconnection and the feedback interconnection of EIP systems results in a EIP system. 
%
When passivity  holds in comparing any two trajectories of $\Sigma$,  the property is called incremental passivity, \cite{Pavlov08}.
\begin{defn}\label{def:IncrP}
System $\Sigma$ (\ref{eq:DynOverall}) is incrementally passive if  there exists a $C^1$, regular, positive semi-definite  storage function $V: \reals^n \times \reals^n \to \reals$ such that for any two inputs $u$, $u'$ and any two solutions $x$, $x'$ corresponding to these inputs, the respective outputs $y$, $y'$ satisfy
		$$\dot{V}(x,x')   \leq  (y-y')^T(u-u')$$
where $\dot{V}=\nabla_x^T V(x,x')  f(x,u) +\nabla_{x'}^T V(x,x') f(x',u')$.
\end{defn}
When $u',x',y'$ are constant (equilibrium conditions), this recovers EIP definition. When system $\Sigma$ is just a static map, incrementally passivity reduces to monotonicity. A static function $y=\Phi(u)$ is EIP if and only if it is incrementally passive, or equivalently, it is monotone. Monotonicity plays an important role in optimization and variational inequalities while passivity plays as critical a role in dynamical systems.

\section{Problem Statement}\label{problem_Statement}
\subsection{Game Formulation}
Consider a set $\mathcal{I}=\{ 1,\dots,N\}$ of $N$ players (agents) involved in a game.  The information sharing between them is described by an undirected  graph $G_c = (\mathcal{I},E)$ or  $G_c$. 
\begin{asmp}\label{asmp:CxnGraph}
	The  communication graph $G_c$ is connected. 
\end{asmp}
Each player $i \in \mathcal{I}$ controls its action $x_i \in \Omega_i$, where $\Omega_i \subseteq \reals^{n_i}$. 
The action set of all players is $\Omega = \prod_{i\in\mathcal{I}}\Omega_i \subseteq \reals^{n}$, $n = \sum_{i\in\mathcal{I}} n_i$.   
Let $x=(x_i,x_{-i})\in {\Omega}$ denote all agents' action profile or $N$-tuple,  where  $x_{-i}\in \Omega_{-i} = \prod_{j\in \mathcal{I}\setminus\{i\}}\Omega_j$ is the $(N-1)$-tuple of all agents' actions except agent $i$'s. Alternatively, $x$ is represented as a stacked vector $x = [x_1^T \dots x_N^T]^T \in \Omega \subseteq \reals^n$. Each player (agent) $i$ aims to minimize its  own cost function $J_i(x_i,x_{-i})$, $J_i : \Omega \to \reals$, which depends on possibly all other players' actions. 
 Let the game thus defined be denoted by $\mathcal{G}(\mathcal{I},J_i,\Omega_i)$.
\begin{defn}\label{defNE}
	Given a game $\mathcal{G}(\mathcal{I},J_i,\Omega_i)$, an action profile $x^* =(x_i^*,x_{-i}^*)\in \Omega$ is  a Nash Equilibrium (NE) of $\mathcal{G}$  if
	\begin{align*}
		(\forall i \in \mathcal{I})(\forall y_i \in \Omega_i) \quad J_i(x_i^*,x_{-i}^*) \leq J_i(y_i,x_{-i}^*)
	\end{align*}
\end{defn}
At a Nash Equilibrium no agent has any incentive to unilaterally deviate from its action. 
In the following we use one of the following two  basic convexity and smoothness assumptions, which ensure the existence of a pure NE.

\begin{asmp} \label{asmp:Jsmooth}

\begin{itemize}
\item [(i)] 	For every $i\in\mathcal{I}$, $\Omega_i=\reals^{n_i}$, the cost function  $J_i:\Omega \to \reals$  is $\mathcal{C}^2$ in its arguments, strictly convex and radially unbounded in $x_i$,  for every $x_{-i}\in {\Omega}_{-i}$.

\item [(ii)] For every $i\in\mathcal{I}$,  $\Omega_i$ is a non empty, convex and compact subset of $\reals^{n_i}$ and the cost function $J_i:\Omega \to \reals$  is $\mathcal{C}^1$ in its arguments and (strictly) convex in $x_i$, for every $x_{-i}\in {\Omega}_{-i}$.
\end{itemize}
\end{asmp}
Under Assumption \ref{asmp:Jsmooth}(i) from Corollary 4.2 in  \cite{B99} it follows that a pure NE $x^*$ exists. Moreover,  an NE satisfies
\begin{align} \label{eq:ViNash_inner}
 \nabla_{i} J_i(x^*_i,x^*_{-i})=0, (\forall i \in \mathcal{I}) \quad \text{or} \quad  F(x^*) = 0 
\end{align}
where $\nabla_{i} J_i(x_i,x_{-i}) =\frac{\partial J_i}{\partial x_i}(x_i,x_{-i})\in \reals^{n_i}$, is the gradient of agent $i$'s cost function $J_i(x_i,x_{-i})$ with respect to  its own action $x_i$ and  $F : \Omega \to \reals^n$ is the pseudo-gradient defined by stacking all agents'  partial gradients, 
\begin{align}\label{eq:expPsuedoGrad_F}
F(x) = [\nabla_{1} J^T_1(x),\dotsc,\nabla_{N} J^T_N(x)]^T
\end{align}

Under Assumption \ref{asmp:Jsmooth}(ii) it follows from Theorem 4.3 in \cite{B99} that a pure NE exists, based on Brower's fixed point theorem. Under just convexity of $J_i$ with respect to $x_i$, existence of an NE follows based on a Kakutani's fixed point theorem. Moreover a Nash equilibrium (NE) $x^*\in \Omega$ satisfies the variational inequality (VI) (cf. Proposition 1.4.2, \cite{FP07}),
\begin{align} \label{eq:ViNash}
	(x-x^*)^TF(x^*)\geq 0 \quad \forall x\in {\Omega}
\end{align}
and projected gradient methods need be used, \cite{FP07}. Additionally (\ref{eq:ViNash}) can be written as $-F(x^*)^T(x-x^*)\leq 0$ and from the definition of the normal cone, 
\begin{align}
	-F(x^*) \in N_{\Omega}(x^*) \label{eq:ViNashNormCone}
\end{align}

Next we state typical assumptions on the pseudo-gradient.

 \begin{asmp} \label{asmp:PseudoGradMono}
\begin{itemize}
\item [(i)] 		The pseudo-gradient $F: \Omega \to \reals^n$ is strictly monotone,
$
		(x-x')^T(F(x)-F(x')) > 0, \, \forall x\neq x' 
$.

\item [(ii)] 	The pseudo-gradient $F: \Omega \to \reals^n$ is strongly monotone,
$
	(x-x')^T(F(x)-F(x')) \geq \mu \| x-x'\|^2,  \,\forall x, x' \in \Omega
$,  
for $\mu >0$, and Lipschitz continuous, $ 
	\| F(x)-F(x')\| \leq \underline{\theta} \| x-x'\|, \, \forall x, x' \in \Omega
$, 
where  $\underline{\theta} >0$.
\end{itemize}
\end{asmp}
Under Assumption \ref{asmp:PseudoGradMono}(i) or \ref{asmp:PseudoGradMono}(ii),  the game has a unique NE, (cf. Theorem 3 in \cite{SFPP14}).

%

The above setting refers to  players' strategic interactions, but it does not specify what knowledge or information each player has. Since $J_i$ depends on all players' actions,  an introspective calculation of an NE requires  \emph{complete information}  where each player knows  the cost functions and strategies of all other players, see Definition \ref{defNE} and (\ref{eq:ViNash_inner}). A game with \emph{incomplete information} refers to players not fully knowing the cost functions or strategies of the others,   \cite{SB87}. Throughout the paper, we assume  $J_i$ is known by player $i$ only.  In a game with incomplete but \emph{perfect information}, each agent has knowledge of the actions of all other players, $x_{-i}$. We refer to the case when players are not able to observe the actions of \emph{all} the other players, as a game with incomplete and  \emph{imperfect or partial information}. This is the setting we consider in this paper: we assume players can communicate only locally, with their neighbours. Our goal is to derive a dynamics for seeking an NE in the \emph{incomplete, partial information case, over a communication graph $G_c$}.  We review first the case of perfect information, and treat the case of partial or networked information in the following sections. In the first part of the paper, for simplicity of the arguments, we consider  $\Omega_i = \reals^{n_i}$ and Assumption \ref{asmp:Jsmooth}(i).  We consider compact action sets and treat the case of projected dynamics, under Assumption \ref{asmp:Jsmooth}(ii), in Section \ref{projection}.

\subsection{Gradient Dynamics with Perfect information} \label{perfectCommunication}
In a game of perfect information, under Assumption \ref{asmp:Jsmooth}(i), a typical gradient-based dynamics, \cite{ArrowHurwitzUzawa1958},  \cite{Flam02}, \cite{SA05}, \cite{FP07}, 
  can be used for each player $i$ 
\begin{align} \label{eq:gradient}
	\mathcal{P}_i : & \quad	\dot{x}_i = - \nabla_i J_i(x_i, x_{-i}), \quad \forall i\in \mathcal{I}
\end{align}
or, $\mathcal{P} : \,\,\, \dot{x} = -F(x)$, the overall  system of all the agents' dynamics stacked together. Assumption \ref{asmp:Jsmooth}(i)  ensures existence and uniqueness of solutions of (\ref{eq:gradient}). 
Note that     (\ref{eq:gradient}) requires all-to-all instantaneous information exchange between players or  over a \emph{complete communication} graph. Convergence of  (\ref{eq:gradient})  is typically shown under strict (strong)  monotonicity  on the pseudo-gradient $F$, \cite{FP07}, \cite{Flam02}, or under strict diagonal dominance of its Jacobian evaluated at $x^*$,\cite{FKB12}. We provide a passivity interpretation.  The gradient dynamics $\mathcal{P}$ (\ref{eq:gradient})  is the feedback interconnection between a bank of $N$ integrators $\Sigma$ and the static pseudo-gradient map $F(\cdot)$, Figure \ref{fig:decomposeGrad}. 
\begin{figure}[h!]
  \centering
\begin{tikzpicture}[auto, node distance=1cm]
    \node [input, name=input] {};
    \node [sum, right = of input] (sum) {};
    \node [block, right = of sum] (dynamics) {$ \Sigma : \begin{cases}
	\dot{x} = u \\
	y = x
  \end{cases}$};
    \node [output, right = of dynamics] (output) {};
    
	\node [block, below = of dynamics] (game) {$F(\cdot)$};

    \draw [->] (sum) -- node {$u$} (dynamics);
    \draw [->] (dynamics) -- node [name=y] {$y$}(output);
        \draw [->] (y) |- node [anchor=south east] {} (game);
    \draw [->] (game) -| node[pos=0.99] {$-$} 
           node [anchor=south west] {} (sum);
\end{tikzpicture}
	\caption{Gradient Dynamics (\ref{eq:gradient}) as Feedback Interconnection of two EIP}
	\label{fig:decomposeGrad}
\end{figure}
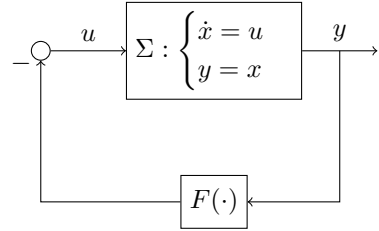
 $\Sigma$  is MEIP with storage function $V(x) = \frac{1}{2}\|x-\overline{x}\|^2$, while  $F(\cdot)$ is static and under Assumption \ref{asmp:PseudoGradMono}(i)  is incrementally passive (EIP). Hence their interconnection  is also EIP and asymptotic stability can be shown using the same storage function. 
\begin{lemma} \label{lemma:perfectInfo}
	Consider a game $\mathcal{G}(\mathcal{I},J_i, \Omega_i)$ in the perfect information case, under Assumptions \ref{asmp:Jsmooth}(i).  Then,  the equilibrium point  $\bar{x}$ of the gradient dynamics, (\ref{eq:gradient}) is the NE of the game $x^*$ and, under Assumption \ref{asmp:PseudoGradMono}(i), is globally asymptotically. Alternatively, under \ref{asmp:PseudoGradMono}(ii) $x^*$  is exponentially stable, hence the solutions of (\ref{eq:gradient}) converge exponentially to the NE of the game, $x^*$. 
\end{lemma}
\begin{proof}At  an equilibrium $\overline{x}$,   of (\ref{eq:gradient}), $F(\overline{x})=0$, hence by (\ref{eq:ViNash_inner}), $\bar{x} =x^*$, the NE of $\mathcal{G}(\mathcal{I},J_i, \Omega_i)$. Consider the quadratic Lyapunov function $V : \Omega \to \reals$, 
$	V(x) = \frac{1}{2} \| x - \bar{x}\|^2$. Along (\ref{eq:gradient}) using  $F(\overline{x})=0$, 
$	\dot{V}(x)  = -(x - \overline{x})^T(F(x) - F(\overline{x})) < 0$, for all $ x \neq  \overline{x}$, by Assumption \ref{asmp:PseudoGradMono}(i). Hence $\dot{V}(x) \leq 0$ and $\dot{V}(x) = 0$ only if $x=\overline{x}=x^*$. Since $V$ is radially unbounded, the conclusion follows by  LaSalle's theorem  \cite{K02}. Under Assumption \ref{asmp:PseudoGradMono}(ii), $	\dot{V}(x) \leq - \mu  \|x - \overline{x} \|^2$, $\forall x$ and global exponential stability follows immediately.
\end{proof}

\section{NE Seeking Dynamics over a Graph}\label{graphCommunication}

In this section we consider the following question: how can we modify the gradient dynamics  (\ref{eq:gradient})  such that it converges to  NE  in a networked information setting, over some connected communication graph $G_c$? 

We propose a new augmented gradient dynamics, derived based on the reformulation as a multi-agent agreement problem between the players. We endow each player (agent) with an auxiliary state  that provides an  \emph{estimate} of all other players' actions.  We design a new signal for each player, based on the relative  feedback from its neighbours,  such that these \emph{estimates} agree one with another. 


 Thus assume that player (agent) $i$ maintains an estimate vector $\mathbf{x}^i = [(\mathbf{x}_1^i)^T,\dotsc,(\mathbf{x}_N^i)^T]^T \in \Omega$ where $\mathbf{x}_j^i$ is player $i$'s estimate of player $j$'s action and $\mathbf{x}_i^i = x_i$ is player $i$'s actual action. $\mathbf{x}_{-i}^i$ represents player $i$'s estimate vector without its own action, $\mathbf{x}_i^i$. All agents' vectors are stacked into a single vector $\mathbf{x} = [(\mathbf{x}^1)^T,\dotsc,(\mathbf{x}^N)^T]^T \in \prod_{i\in\mathcal{I}} \Omega = \Omega^N =\reals^{Nn}$.  Note that  the state space  is now $\Omega^N = \prod_{i\in\mathcal{I}} \Omega = \reals^{Nn}$. 
  %
In the enlarged space the estimate components will be different initially, but in the limit all players estimate vectors should be in consensus. 
We modify the gradient dynamics such that player $i$ updates $\mathbf{x}_i^i$ to reduce its own cost function and updates $\mathbf{x}^i_{-i}$ to reach a consensus with the other players. Let  each player combine its gradient-type  dynamics with an integrator-type auxiliary dynamics, driven by some control signal, 
\begin{align}
	\label{eq:baseAgentExpFiner}
	\widetilde{\Sigma}_i &: \begin{cases}
		\begin{bmatrix}
			\dot{x}_i \\
			\dot{\mathbf{x}}^i_{-i}
		\end{bmatrix}
		= \begin{bmatrix}
			-\nabla_i J_i(x_i, \mathbf{x}^i_{-i})  \\
			 0
		\end{bmatrix} + B^i  \mathbf{u}_i\\
		\mathbf{y}_i = (B^i)^T\mathbf{x}^i
	\end{cases}
\end{align} 
where $B^i$ is a full rank ${n\times n}$ matrix. For each player,  $\mathbf{u}_i \in \reals^{n}$ is to be designed based on the relative output feedback from its neighbours,  such that $ \mathbf{x}^i = \mathbf{x}^j$, for all $i,j$, and converge to the NE $x^*$.

Thus we have reformulated the design of NE dynamics over $G_c$ as a multi-agent agreement problem. We note that agent dynamics (\ref{eq:baseAgentExpFiner}) are  heterogenous, separable, but  do not satisfy an individual passivity property as typically assumed in multi-agent literature, e.g. \cite{BP15}, \cite{A07}.  We show next that  a  Laplacian-type feedback can be designed  under strict incremental passivity  of the pseudo-gradient.  

To proceed, we first analyze properties of $\widetilde{\Sigma}_i$ and the overall agents' dynamics $\widetilde{\Sigma}$. Write $\widetilde{\Sigma}_i$ (\ref{eq:baseAgentExpFiner}) in a compact form
\begin{align}
 \label{eq:baseAgentExp} 
	\widetilde{\Sigma}_i &: \begin{cases}
		\dot{\mathbf{x}}^i = -\mathcal{R}_i^T \nabla_i J_i(\mathbf{x}^i) + B^i\mathbf{u}_i \\
		\mathbf{y}_i = (B^i)^T\mathbf{x}^i
	\end{cases}
\end{align}\vspace{-0.2cm}
where 
\begin{align} \label{eq:actualStratREMatrix}
\mathcal{R}_i = [\mathbf{0}_{n_i\times n_{<i}} \,\, I_{n_i} \, \, \mathbf{0}_{n_i\times n_{>i}}] 
\end{align}
and  $n_{<i} = \sum_{j<i\ i,j\in\mathcal{I}} n_j$, $n_{>i} = \sum_{j>i\ i,j\in\mathcal{I}} n_j$. 

Thus $\mathcal{R}_i^T \in \reals^{n\times n_i}$ aligns the gradient to the action component in $\dot{\mathbf{x}}_i$.
From (\ref{eq:baseAgentExp}), with $\mathbf{x} = [(\mathbf{x}^1)^T,\dotsc,(\mathbf{x}^N)^T]^T $, $\mathbf{u} =[\mathbf{u}^T_1,\dots,\mathbf{u}^T_N]^T \in \reals^{Nn}$, the overall agents' dynamics denoted by $\widetilde{\Sigma}$ can be written in stacked form as
\begin{align}\label{eq:baseAgent}
	\widetilde{\Sigma} : \begin{cases}
		\dot{\mathbf{x}} = -\mathcal{R}^T\mathbf{F}(\mathbf{x}) + B\mathbf{u} \\
		\mathbf{y} = B^T\mathbf{x}
	\end{cases}
\end{align}
where  $\mathcal{R} = diag(\mathcal{R}_1, \dots, \mathcal{R}_N)$, $B =diag(B^1,\dots,B^N)$ and \begin{align} \label{eq:expPsuedoGrad}
\begin{aligned}
	\mathbf{F}(\mathbf{x}) &= [\nabla_1 J^T_1(\mathbf{x}^1),\dotsc,\nabla_N J^T_N(\mathbf{x}^N)]^T
\end{aligned}
\end{align}
is the continuous extension of the pseudo-gradient $F$, (\ref{eq:expPsuedoGrad_F}) to the augmented space, $\mathbf{F}(\mathbf{x}) : \Omega^N \to \reals^n$. Note that  $\mathbf{F}(\mathbf{1}_N \otimes x) = F(x)$. In the rest of the paper we consider one of the following two assumptions on the extended $\mathbf{F}$.
\begin{asmp} \label{asmp:strongPseudo}

\begin{itemize}
\item [(i)] 	The extended pseudo-gradient $\mathbf{F}$ is monotone, 
$
		(x -x')^T(\mathbf{F}(\mathbf{x})-\mathbf{F}(\mathbf{x}')) \geq 0, \, \forall \mathbf{x}, \mathbf{x}' \in \Omega^N
$.

\item [(ii)] 	The extended pseudo-gradient $\mathbf{F}$ is Lipschitz continuous, 
$ 
		\| \mathbf{F}(\mathbf{x})-\mathbf{F}(\mathbf{x}') \| \leq  \theta \| \mathbf{x}-\mathbf{x}' \|, \,\forall \mathbf{x}, \mathbf{x}' \in  \Omega^N 
$
where $\theta >0$.
\end{itemize}
\end{asmp}
\begin{remark}
We compare this assumption to similar ones used  in distributed optimization and in  multi-agent coordination control, respectively.  
First, note that Assumption \ref{asmp:strongPseudo}(i) on the extended pseudo-gradient $\mathbf{F}$ holds under individual joint convexity of  each $J_i$ with respect to the full argument. In distributed optimization problems, each objective function is assumed to be strictly (strongly) jointly convex  in the full vector $\mathbf{x}$ and its gradient to be Lipschitz continuous, e.g. \cite{WE11}. Similarly, in multi-agent coordination control, it is standard to assume that  individual agent dynamics are separable and strictly (strongly) incrementally passive, e.g. \cite{BP15}. However, in a game context the individual joint convexity of  $J_i$ with respect to the full argument is too restrictive, unless we have a trivial game with separable cost functions. In general, $J_i$ is coupled to other players' actions while each player has under its control only its own action. This is a key difference versus distributed optimization or multi-agent coordination, one which introduces technical challenges. However, we show that under the monotonicity Assumption \ref{asmp:strongPseudo}(i) on $\mathbf{F}$, the overall  $\widetilde{\Sigma}$ (\ref{eq:baseAgent}) is incrementally passive, hence EIP. Based on this,  we design a new dynamics which converges over \emph{any connected} $G_c$ (Theorem \ref{thm:strongPseudoComm}).   Under the weaker Lipschitz Assumption \ref{asmp:strongPseudo}(ii) on $\mathbf{F}$ and  Assumption \ref{asmp:PseudoGradMono}(ii) on $F$,  we show that the new dynamics  converges over \emph{any sufficiently connected} $G_c$ (Theorem \ref{thm:LipPseudoComm}). 
We also note that Assumption \ref{asmp:strongPseudo}(i) is similar to those used in  \cite{FredIFAC2017},  \cite{WeiShiACC2017}, while Assumption \ref{asmp:strongPseudo}(ii) is weaker.  Assumption \ref{asmp:strongPseudo}(i) is the extension of Assumption \ref{asmp:PseudoGradMono}(i) to the  augmented space, for local communication over the connected graph $G_c$. The weaker Assumption \ref{asmp:strongPseudo}(ii) on $\mathbf{F}$,  is the extension of Lipschitz continuity of $F$ in Assumption \ref{asmp:PseudoGradMono}(ii). We also note that these assumptions could be relaxed to hold only locally around $x^*$ in which case all results become local.
\end{remark}

\begin{lemma}  \label{lemma:ForwardEIP}
	Under Assumption \ref{asmp:strongPseudo}(i), the overall system $\widetilde{\Sigma}$, (\ref{eq:baseAgent}), is incrementally passive, hence EIP.
\end{lemma}
\begin{proof}
Consider two inputs  $\mathbf{u}$, $\mathbf{u}'$ and  let $\mathbf{x}$, $\mathbf{x}'$, $\mathbf{y}$, $\mathbf{y}'$ be the trajectories and outputs of $\widetilde{\Sigma}$ (\ref{eq:baseAgent}). Let the storage function be $V(\mathbf{x},\mathbf{x}') = \frac{1}{2}\|\mathbf{x}-\mathbf{x}'\|^2$. Then, along solutions of (\ref{eq:baseAgent}),
\begin{align}\label{eq:dotV_EIP}
\dot{V} & =-(\mathbf{x}-\mathbf{x}')^T\left [\mathcal{R}^T(\mathbf{F}(\mathbf{x}) - \mathbf{F}(\mathbf{x}')) + B(\mathbf{u} - \mathbf{u}')\right ]  \notag\\
 &= -(x-x')^T \, (\mathbf{F}(\mathbf{x}) - \mathbf{F}(\mathbf{x}')) +(\mathbf{y} - \mathbf{y}')^T(\mathbf{u} - \mathbf{u}') 
\end{align}
by using $\mathcal{R} = diag(\mathcal{R}_1, \dots, \mathcal{R}_N)$ and (\ref{eq:actualStratREMatrix}). Using  Assumption \ref{asmp:strongPseudo}(i) it follows that 
\begin{align*}
\dot{V} \leq  (\mathbf{y} - \mathbf{y}')^T(\mathbf{u} - \mathbf{u}') 
\end{align*}
Thus by Definition \ref{def:IncrP}, $\widetilde{\Sigma}$, is incrementally passive, hence EIP.
\end{proof}

\subsection{Distributed feedback design} \label{commGradDyn}

Given agent dynamics  $\widetilde{\Sigma}_i$, (\ref{eq:baseAgentExp}),  for each individual player we design  $\mathbf{u}_i \in \reals^{n}$ based on the relative output feedback from its neighbours,  such that the auxiliary state variables (estimates) agree one with another and converge to the NE $x^*$. For simplicity, take  $B = I_{Nn}$ so that $\mathbf{y}  =\mathbf{x}$. 

Let $\mathcal{N}_i$ denote the set of neighbours of player $i$ in graph $G_c$ and $L$ denote the symmetric Laplacian matrix. Let $\mathbf{L} = L\otimes I_n$ denote the augmented Laplacian matrix which satisfies 
\begin{align}\label{eq:nullExpL}
	Null(\mathbf{L}) = Range(\mathbf{1}_N \otimes I_n)	
\end{align}
and $Range(\mathbf{L}) = Null(\mathbf{1}_N^T \otimes I_n) $, based on   $L\mathbf{1}_N=\mathbf{0}_N$. For any $ W\in \reals^{q\times n}$,  and any $\mathbf{x} \in \reals^{Nn}$, using  $L\mathbf{1}_N = \mathbf{0}_N$,
\begin{align}
	(\mathbf{1}_N^T\otimes W)\mathbf{L} \mathbf{x} &=( (\mathbf{1}_N^TL)\otimes(W I_n))\mathbf{x} = \mathbf{0}_q \label{eq:LIeqZero}
\end{align}

With respect to the overall dynamics  $\widetilde{\Sigma}$, (\ref{eq:baseAgent}), the objective is to design $\mathbf{u}$ such that $\mathbf{x} $ reaches consensus, i.e., $\mathbf{1}_N \otimes x$, for some $x \in \Omega$ and $x$ converges towards the NE $x^*$. The consensus  condition is written as $\mathbf{L}\mathbf{x}=0$.  Since   $\widetilde{\Sigma}$, (\ref{eq:baseAgent}) is incrementally passive by Lemma \ref{lemma:ForwardEIP}, and $\mathbf{L}$ is positive semi-definite, a passivity-based control design,  e.g. \cite{A07}, suggests taking $\mathbf{u} = -\mathbf{L} \mathbf{x}$. The resulting closed-loop system which represents the new overall system dynamics $\widetilde{\mathcal{P}}$ is given in stacked notation as
 \begin{align} \label{eq:overallDyn}
	\widetilde{\mathcal{P}} :& \quad 
 		\dot{\mathbf{x}} = -\mathcal{R}^T\mathbf{F}(\mathbf{x}) - \mathbf{L}\mathbf{x}
\end{align}
shown in Figure \ref{fig:expFeedback} as the  feedback interconnection between  $\widetilde{\Sigma}$ and  $\mathbf{L}$. 
Local solutions of  (\ref{eq:overallDyn}) exist by Assumption \ref{asmp:Jsmooth}(i). 
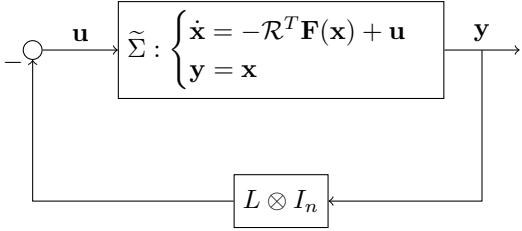
\begin{figure}[h!]
  \centering
\begin{tikzpicture}[auto, node distance=1cm]
    \node [input, name=input] {};
    \node [sum, right = of input] (sum) {};
    \node [block, right = of sum] (dynamics) {$	\widetilde{\Sigma} : 				\begin{cases}
		\dot{\mathbf{x}} = -\mathcal{R}^T\mathbf{F}(\mathbf{x}) + \mathbf{u}
		 \\
		\mathbf{y} = \mathbf{x}
	\end{cases}$};
    \node [output, right = of dynamics] (output) {};    
	\node [block, below = of dynamics] (game) {$ L \otimes I_n$};

    \draw [->] (sum) -- node {$\mathbf{u}$} (dynamics);
    \draw [->] (dynamics) -- node [name=y] {$\mathbf{y}$}(output);
    \draw [->] (y) |- node [anchor=south east] {} (game);
    \draw [->] (game) -| node[pos=0.99] {$-$} 
        node [anchor=south west] {} (sum);
\end{tikzpicture}
	\caption{Augmented gradient dynamics (\ref{eq:overallDyn}) over $G_c$}
	\label{fig:expFeedback}
\end{figure}
The new individual player dynamics $\widetilde{\mathcal{P}}_i$ are
 \begin{align}
	\widetilde{\mathcal{P}}_i :& \quad 
	\dot{\mathbf{x}}^i = -\mathcal{R}_i^T\nabla_i J_i(\mathbf{x}^i) - \sum_{j\in\mathcal{N}_i}( \mathbf{x}^i - \mathbf{x}^j ) \label{eq:expandedIndividual}
\end{align}or, separating the action $\mathbf{x}^i_i=x_i$ and  estimate $\mathbf{x}^i_{-i}$  dynamics,
\begin{align} \label{eq:agentActionEstimate}
	\widetilde{\mathcal{P}}_i : \begin{cases}
	\begin{bmatrix}
	\dot{x}_i \\
	\dot{\mathbf{x}}^i_{-i}
	\end{bmatrix} =
	\begin{bmatrix}
	-\nabla_i J_i(x_i,\mathbf{x}^i_{-i}) - \mathcal{R}_i\sum_{j\in\mathcal{N}_i} ( \mathbf{x}^i - \mathbf{x}^j) \\
	-\mathcal{S}_i(\sum_{j\in\mathcal{N}_i} \mathbf{x}^i - \mathbf{x}^j)
	\end{bmatrix}
	\end{cases}
\end{align}\vspace{-0.2cm}
where  \begin{align}\label{eq:actualSMatrix}
\mathcal{S}_i &= \begin{bmatrix}
	I_{n<i} & \mathbf{0}_{n<i \times n_i} & \mathbf{0}_{n<i \times n>i} \\
	\mathbf{0}_{n>i \times n<i} & \mathbf{0}_{n>i \times n_i} & I_{n>i} 
\end{bmatrix}\end{align}
 and $\mathcal{S}_i$ removes $\mathbf{x}^i_i=x_i$, its own action component from agent $i$'s estimate vector, $\mathbf{x}^i$. 

For player $i$, $\widetilde{\mathcal{P}}_i$, (\ref{eq:expandedIndividual}) or (\ref{eq:agentActionEstimate}) is  clearly distributed over $G_c$. Its input is the relative difference between its estimate and its neighbours'. In standard consensus terms, agent $i$ can use this information to move in the direction of the average value of its neighbours, while the gradient term enforces the move towards minimizing its own cost. 
Compared to the gossip-based algorithm in \cite{FPAuto2016},  the action part of  (\ref{eq:agentActionEstimate}) has an extra correction term. This term is instrumental in proving convergence on a single timescale as  shown in Section \ref{convergence}.

The next result shows that the equilibrium of (\ref{eq:overallDyn}) or (\ref{eq:agentActionEstimate})  occurs when the agents are at a consensus and at NE. 
\begin{lemma} \label{lemma:eqIsConsensus}
	Consider  a game $\mathcal{G}(\mathcal{I},J_i,\Omega_i)$ over a communication graph $G_c$ under Assumptions \ref{asmp:CxnGraph}, \ref{asmp:Jsmooth}(i). Let the dynamics for each agent  $\widetilde{\mathcal{P}}_i$ be as in  (\ref{eq:expandedIndividual}), (\ref{eq:agentActionEstimate}), or overall $\widetilde{\mathcal{P}}$,  (\ref{eq:overallDyn}). At an equilibrium point $\overline{\mathbf{x}}\in {\Omega}^N$ the estimate vectors of all players are equal $\bar{\mathbf{x}}^i=\bar{\mathbf{x}}^j$, $\forall i,j\in\mathcal{I}$ and equal to the Nash equilibrium profile $x^*$, hence the action components of all players  coincide with the optimal actions, $\bar{\mathbf{x}}^i_i = x^*_i$, $\forall i\in\mathcal{I}$.
\end{lemma}
\begin{proof}
Let $\overline{\mathbf{x}}$ denote an equilibrium of  (\ref{eq:overallDyn}),  
\begin{align}
\label{eq:LeqPoint}
\mathbf{0}_{Nn} &= -\mathcal{R}^T\mathbf{F}(\bar{\mathbf{x}}) - \mathbf{L}\bar{\mathbf{x}}
\end{align}
Pre-multiplying both sides by $(\mathbf{1}_N^T\otimes I_n)$, yields 
$
\mathbf{0}_{n} = -(\mathbf{1}_N^T\otimes I_n)\mathcal{R}^T\mathbf{F}(\bar{\mathbf{x}})
$,  where $(\mathbf{1}_N^T\otimes I_n) \mathbf{L}\bar{\mathbf{x}} =0$ was used by (\ref{eq:nullExpL}).  Using (\ref{eq:actualStratREMatrix}) and simplifying  $(\mathbf{1}_N^T\otimes I_n)\mathcal{R}^T$ gives
\begin{align}
\mathbf{0}_{n} &= \mathbf{F}(\bar{\mathbf{x}}), \qquad \text{or} \qquad \nabla_i J_i(\bar{\mathbf{x}}^i) =0, \,\, \forall i \in\mathcal{I} \label{eq:EqExPsGrad}
\end{align}
by (\ref{eq:expPsuedoGrad}).  Substituting (\ref{eq:EqExPsGrad}) into (\ref{eq:LeqPoint}) results in
$\mathbf{0}_{Nn} = - \mathbf{L}\bar{\mathbf{x}}$. 
From this it follows that $\bar{\mathbf{x}}^i = \bar{\mathbf{x}}^j$, $\forall i,j\in\mathcal{I}$ by Assumption \ref{asmp:CxnGraph} and  (\ref{eq:nullExpL}). Therefore   $\bar{\mathbf{x}} = \mathbf{1}_N\otimes \bar{x}$, for some $ \bar{x} \in {\Omega}$. 
Substituting this back into (\ref{eq:EqExPsGrad}) yields 
$	\mathbf{0}_{n} = \mathbf{F}(\mathbf{1}_N\otimes \bar{x})$
or $\nabla_i J_i(\bar{x}) =0$, for  all $ i \in\mathcal{I}$. 
Using (\ref{eq:expPsuedoGrad}), $\nabla_i J_i(\bar{x}_i,\bar{x}_{-i}) =0$, for all $i \in\mathcal{I}$, or 
$
	\mathbf{0}_{n} = F(\bar{x}) \label{eq:expandNECondition}
$. 
Therefore by (\ref{eq:ViNash_inner}) $\bar{x} =x^*$, hence $\bar{\mathbf{x}} = \mathbf{1}_N\otimes x^*$ and for all $i,j \in\mathcal{I}$, $\bar{\mathbf{x}}^i = \bar{\mathbf{x}}^j=x^*$ the NE of the game. 
\end{proof}

\section{Convergence Analysis}\label{convergence}

In this section we analyze the convergence of player's new dynamics $\widetilde{\mathcal{P}}_i$  (\ref{eq:expandedIndividual}), (\ref{eq:agentActionEstimate}) or overall $\widetilde{\mathcal{P}}$   (\ref{eq:overallDyn})  to the NE of the game, over a connected graph $G_c$.     We consider two cases. 

In  Section \ref{sec:singleTime} we analyze convergence of (\ref{eq:agentActionEstimate}) on a single timescale:  in Theorem \ref{thm:strongPseudoComm} under Assumptions \ref{asmp:CxnGraph}, \ref{asmp:Jsmooth}(i),  \ref{asmp:PseudoGradMono}(i) and \ref{asmp:strongPseudo}(i), and  in Theorem \ref{thm:LipPseudoComm} under Assumptions \ref{asmp:CxnGraph}, \ref{asmp:Jsmooth}(i), \ref{asmp:PseudoGradMono}(ii) and \ref{asmp:strongPseudo}(ii). We exploit  the incremental passivity (EIP) property of $\widetilde{\Sigma}$ (\ref{eq:baseAgent}) (Lemma \ref{lemma:ForwardEIP}) and diffusive properties of the Laplacian. 

In Section \ref{sec:twoTime}, we modify the estimate component of the dynamics (\ref{eq:agentActionEstimate}) to be much faster, and in Theorem \ref{thm:STRGMonPseudoComm} prove convergence under  Assumptions  \ref{asmp:CxnGraph}, \ref{asmp:Jsmooth}(i), \ref{asmp:PseudoGradMono}(ii) and \ref{asmp:strongPseudo}(ii), using a two-timescale singular perturbation approach.

\subsection{Single-Timescale  Consensus and Player Optimization}\label{sec:singleTime}

Theorem \ref{thm:strongPseudoComm} shows that, under Assumption \ref{asmp:strongPseudo}(i), (\ref{eq:agentActionEstimate})  converges to the NE of the game, over \emph{any connected} $G_c$.
\begin{thm} \label{thm:strongPseudoComm}
Consider a game $\mathcal{G}(\mathcal{I},J_i, \Omega_i)$ over a communication graph $G_c$ under Assumptions \ref{asmp:CxnGraph}, \ref{asmp:Jsmooth}(i), \ref{asmp:PseudoGradMono}(i) and \ref{asmp:strongPseudo}(i).  Let each player's dynamics $\widetilde{\mathcal{P}}_i$, be as in (\ref{eq:expandedIndividual}), (\ref{eq:agentActionEstimate}), or overall $\widetilde{\mathcal{P}}$,   (\ref{eq:overallDyn}), as in Figure \ref{fig:expFeedback}.  Then, any solution of (\ref{eq:overallDyn}) is bounded and asymptotically converges to $\mathbf{1}_N\otimes x^*$, and the actions components converge to the NE of the game, $x^*$. 
\end{thm}

\begin{proof}
By Lemma \ref{lemma:eqIsConsensus}, the  equilibrium   of  $\widetilde{\mathcal{P}}$ (\ref{eq:overallDyn}) is  $\bar{\mathbf{x}} = \mathbf{1}_N\otimes x^*$. 
 We consider the quadratic storage function $V(\mathbf{x}) = \frac{1}{2}\|\mathbf{x}-\bar{\mathbf{x}} \|^2$, $\bar{\mathbf{x}} = \mathbf{1}_N\otimes x^*$ as a Lyapunov function. As in (\ref{eq:dotV_EIP}) in the proof of  Lemma \ref {lemma:ForwardEIP},  using  $\mathbf{u} = -\mathbf{L}\mathbf{x}$, (\ref{eq:LeqPoint}) we obtain that  along the solutions of  (\ref{eq:overallDyn}), 
\begin{align}\label{eq:dotSfull}
	\dot{V} = -(\mathbf{x}-\bar{\mathbf{x}})^T\mathcal{R}^T(\mathbf{F}(\mathbf{x}) - \mathbf{F}(\bar{\mathbf{x}}))-(\mathbf{x}-\bar{\mathbf{x}})^T\mathbf{L}(\mathbf{x}-\bar{\mathbf{x}}) 
\end{align}
where $\bar{\mathbf{x}} = \mathbf{1}_N\otimes x^*$, and  $\mathbf{x}^T\mathcal{R}^T =x$. By Assumption \ref{asmp:strongPseudo}(i)   and since the augmented Laplacian $\mathbf{L}$ is positive semi-definite it follows that  $\dot{V} \leq 0$, for all  $\mathbf{x}\in {\Omega}^N$, hence all trajectories of  (\ref{eq:overallDyn}) are bounded and $\overline{\mathbf{x}}$ is stable. To show convergence we resort to  LaSalle's invariance principle, \cite{K02}. 

From (\ref{eq:dotSfull}), $\dot{V} =0$ when both terms in (\ref{eq:dotSfull}) are zero, i.e., $(\mathbf{x}-\bar{\mathbf{x}})^T\mathcal{R}^T(\mathbf{F}(\mathbf{x}) - \mathbf{F}(\bar{\mathbf{x}}))=0$ and $(\mathbf{x}-\bar{\mathbf{x}})^T\mathbf{L}(\mathbf{x}-\bar{\mathbf{x}}) =0$.  By Assumption \ref{asmp:CxnGraph} and  (\ref{eq:nullExpL}), $(\mathbf{x}-\bar{\mathbf{x}})^T\mathbf{L}(\mathbf{x}-\bar{\mathbf{x}}) =0$ 
is equivalent to  $\mathbf{x}-\bar{\mathbf{x}} = \mathbf{1}_N\otimes \underline{x}$, for some $\underline{x} \in \reals^n$. Since at equilibrium $\bar{\mathbf{x}} = \mathbf{1}_N\otimes x^*$, this implies that  $\mathbf{x}=\mathbf{1}_N\otimes x$, for some $x\in \reals^n$.   By  (\ref{eq:expPsuedoGrad}) $\mathbf{F}(\mathbf{1}_N\otimes x) = F(x)$. 
Using $\mathcal{R}$ in (\ref{eq:actualStratREMatrix})  yields for the first term in (\ref{eq:dotSfull}) to be zero, 
\begin{align}\label{eq:dotSconsensus}
0    
	&= -(\mathbf{1}_N\otimes x - \mathbf{1}_N\otimes x^*)^T\mathcal{R}^T[F(x)-F(x^*)]  \notag \\
&= -(x-x^*)^T[F(x)-F(x^*)] < 0 \quad \forall x\neq x^*
\end{align}
where strict inequality follows by Assumption \ref{asmp:PseudoGradMono}(i). Therefore $\dot{V}=0$ in  (\ref{eq:dotSfull}) only if $x = x^*$ and  hence $\mathbf{x} =\mathbf{1}_N\otimes x^*$.  Since $V$ is radially unbounded,  the conclusion follows by LaSalle's invariance principle.
\end{proof}
If  $ \mathbf{F}(\cdot)$ is strongly monotone, exponential convergence can be shown over \emph{any connected} $G_c$. Next we show that,  under a weaker Lipschitz property of $\mathbf{F}$ (Assumption \ref{asmp:strongPseudo}(ii)) and  strong monotonicity  of $F$ (Assumption \ref{asmp:PseudoGradMono}(ii)),  (\ref{eq:agentActionEstimate}) converges  over \emph{any sufficiently connected} $G_c$. 
\begin{thm} \label{thm:LipPseudoComm}
Consider a game $\mathcal{G}(\mathcal{I},J_i, \Omega_i)$ over a communication graph $G_c$ under Assumptions \ref{asmp:CxnGraph}, \ref{asmp:Jsmooth}(i), \ref{asmp:PseudoGradMono}(ii) and  \ref{asmp:strongPseudo}(ii).  Let each player's dynamics $\widetilde{\mathcal{P}}_i$ be as in (\ref{eq:expandedIndividual}), (\ref{eq:agentActionEstimate}) or overall $\widetilde{\mathcal{P}}$   (\ref{eq:overallDyn}).  Then, if 
$\lambda_2(L)  >\frac{\theta^2}{\mu} + \theta$,  
any solution of (\ref{eq:overallDyn})  converges asymptotically to $\mathbf{1}_N\otimes x^*$, and the actions components converge to the NE of the game, $x^*$. If $\lambda_2(L)  >\frac{N \theta^2}{\mu} + \theta$, then convergence is exponential.  
\end{thm}
\begin{proof} We decompose $\reals^{Nn}$ as $\reals^{Nn} = C_N^n \oplus E_N^n$, into the  consensus subspace $C_N^n =\{ \mathbf{1}_N\otimes x  \, | x \in \reals^n \}$ and its orthogonal complement  $E_N^n $. Let  two projection matrices  be defined as
$$P_C = \frac{1}{N}  \mathbf{1}_N \otimes  \mathbf{1}_N^T \otimes I_n, \, \,
P_E = I_{Nn} -  \frac{1}{N}  \mathbf{1}_N \otimes  \mathbf{1}_N^T \otimes I_n
$$
Then any $\mathbf{x}  \in \reals^{Nn}$ can be decomposed as   $ \mathbf{x} = \mathbf{x}^{||} \, + \, \mathbf{x}^{\perp}$, where  $\mathbf{x}^{||}  =  P_C \mathbf{x} \in  C_N^n$ and $\mathbf{x}^{\perp}=  P_E \mathbf{x} \in  E_N^n$, with $ ( \mathbf{x}^{||})^T\mathbf{x}^{\perp} =0$. Thus $\mathbf{x}^{||} = \mathbf{1}_N \otimes x$, for some $x \in \reals^n$, so that  $\mathbf{L} \mathbf{x}^{||} =0$, and $min_{\mathbf{x}^{\perp}\in E_N^n} (\mathbf{x}^{\perp})^T \mathbf{L} \mathbf{x}^{\perp} = \lambda_2(L) \| \mathbf{x}^{\perp} \|^2$, $ \lambda_2(L) >0$.

Consider  $V(\mathbf{x}) =  \frac{1}{2}\|\mathbf{x}-\overline{\mathbf{x}} \|^2 $, $\overline{\mathbf{x}}= \mathbf{1}_N\otimes x^*$,  which using  $ \mathbf{x} = \mathbf{x}^{||} \, + \, \mathbf{x}^{\perp}$  can be written as $V(\mathbf{x}) = \frac{1}{2}\|\mathbf{x}^{\perp}\|^2 +  \frac{1}{2}\|\mathbf{x}^{||}- \overline{\mathbf{x}} \|^2$. 

Then following the same steps as in  Lemma \ref {lemma:ForwardEIP}, and replacing  $\mathbf{x}$ with its decomposed components $\mathbf{x}^\perp, \mathbf{x}^{||}$ a relation similar to (\ref{eq:dotSfull}) follows  along   (\ref{eq:overallDyn}), i.e.,
\begin{align*}
	\dot{V} &\leq -(\mathbf{x}-\overline{\mathbf{x}})^T\mathcal{R}^T\left [\mathbf{F}(\mathbf{x}) - \mathbf{F}(\overline{\mathbf{x}})\right ] \\
	&\quad - (\mathbf{x}-\overline{\mathbf{x}})^T\mathbf{L}(\mathbf{x} - \overline{\mathbf{x}})
\end{align*}
Using  $\mathbf{x}=\mathbf{x}^\perp + \mathbf{x}^{||}$, $\overline{\mathbf{x}}=\mathbf{1}_N\otimes x^*$, $\mathbf{L} \mathbf{x}^{||} =0$, $\dot{V}$ can be written as
\begin{align*}
\dot{V} &\leq -(\mathbf{x}^\perp)^T\mathcal{R}^T\left [\mathbf{F}(\mathbf{x}) - \mathbf{F}(\mathbf{x}^{||})\right ] \\
	& -(\mathbf{x}^\perp)^T\mathcal{R}^T\left [\mathbf{F}(\mathbf{x}^{||}) - \mathbf{F}(\overline{\mathbf{x}})\right ] \\
	& - (\mathbf{x}^{||} - \mathbf{1}_N\otimes x^* )^T \mathcal{R}^T \left [\mathbf{F}(\mathbf{x}) - \mathbf{F}(\mathbf{x}^{||})\right ]\\
	& - (\mathbf{x}^{||} - \mathbf{1}_N\otimes x^* )^T \mathcal{R}^T \left [\mathbf{F}(\mathbf{x}^{||}) - \mathbf{F}(\overline{\mathbf{x}})\right ]\\
	& - (\mathbf{x}^\perp)^T\mathbf{L}\mathbf{x}^\perp
\end{align*}
Using  $ (\mathbf{x}^{\perp})^T \mathbf{L} \mathbf{x}^{\perp} \geq \lambda_2(L) \| \mathbf{x}^{\perp} \|^2$ and $\mathbf{F}(\mathbf{x}^{||}) =F(x)$, 
  $\mathbf{F}(\overline{\mathbf{x}}) =F(x^*)$ 
    yields 
\begin{align*}
\dot{V} &\leq \|\mathbf{x}^\perp \| \| \mathbf{F}(\mathbf{x}) - \mathbf{F}(\mathbf{x}^{||}) \| \\
	& -(\mathbf{x}^\perp)^T\mathcal{R}^T\left [F(x)- F(x^*)\right ] \\
	& - (\mathbf{x}^{||} - \mathbf{1}_N\otimes x^* )^T \mathcal{R}^T \left [\mathbf{F}(\mathbf{x}) - \mathbf{F}(\mathbf{x}^{||})\right ]\\
	& - (\mathbf{x}^{||} - \mathbf{1}_N\otimes x^* )^T \mathcal{R}^T \left [F(x) -F(x^*)\right ]\\
	& -  \lambda_2(L) \| \mathbf{x}^\perp \|^2
\end{align*}
Under Assumption   \ref{asmp:strongPseudo}(ii), $ \| \mathbf{F}(\mathbf{x}) - \mathbf{F}(\mathbf{x}^{||}) \|  \leq    \theta \| \mathbf{x}^{\perp} \| $, so that, after simplifying $\mathcal{R}\mathbf{x}^{||} = x $, $ \mathcal{R}( \mathbf{1}_N\otimes x^* )= x^*$, 
\begin{align*} 
\dot{V} & \leq  -(\lambda_2(L) -\theta) \|\mathbf{x}^{\perp}\|^2 - \mathbf{x}^{\perp} \mathcal{R}^T  \left [F(x)- F(x^*)\right ]\\
& -(x -  x^* )^T [ \mathbf{F}(\mathbf{x})- \mathbf{F}(\mathbf{x}^{||}) ] \\
& - (x -  x^* )^T  [ F(x) - F(x^*) ]
\end{align*}


%

Using again Assumption \ref{asmp:strongPseudo}(ii) in the $3^{rd}$   and $2^{nd}$ terms and Assumption \ref{asmp:PseudoGradMono}(ii) in the $4^{th}$ one,  it can be shown that 
\begin{align}\label{eq:dotVstrong_asy}
\dot{V} & \leq  -(\lambda_2(L) -\theta) \|\mathbf{x}^{\perp}\|^2 +2 \theta \| \mathbf{x}^{\perp} \| \| x - x^*\| \\
&
- \mu  \| x - x^*\| ^2 \notag
\end{align} 
or $
\dot{V}  \leq  - [ \| x - x^*\|  \,\, \, \, \,  \|\mathbf{x}^{\perp}\| ] \, \Theta  \, [  \| x - x^*\| \,\, \, \, \, \|\mathbf{x}^{\perp}\| ]^T
$, 
where $\Theta = \left [ \begin{array}{cc}  \mu & -\theta \\ -\theta & \lambda_2(L) - \theta \end{array} \right ]$. Under the conditions in the statement, $\Theta$ is positive definite. Hence $\dot{V}\leq 0$ and $\dot{V}=0$ only if $\mathbf{x}^{\perp}=0$ and  $x=x^*$, hence $\mathbf{x} = \mathbf{1}_N\otimes x^*$. The  conclusion follows  by LaSalle's invariance principle. 

Exponential convergence follows   using $ \| x - x^*\| = \frac{1}{\sqrt{N}} \| \mathbf{x}^{||} - \overline{\mathbf{x}} \|$, under the stricter condition on $\lambda_2(L)$.  Indeed, 
$$
\dot{V}  \leq  -  \left [ \begin{array}{cc} \|  \mathbf{x}^{||} - \overline{\mathbf{x}}\|  &  \|\mathbf{x}^{\perp}\|  \end{array} \right ]   \, \Theta_N  \, \left [ \begin{array}{c} 
 \|  \mathbf{x}^{||} - \overline{\mathbf{x}} \| \\  \|\mathbf{x}^{\perp}\|  \end{array} \right ]  
$$where $\Theta_N = \left [ \begin{array}{cc}  \frac{1}{N} \mu & -\theta \\ -\theta & \lambda_2(L) - \theta \end{array} \right ] $ is positive definite if $\lambda_2(L)  >\frac{N \theta^2}{\mu} + \theta$. This implies that  $\dot{V}(\mathbf{x}(t)) \leq - \eta V(\mathbf{x}(t))$,  for some $\eta >0$, so that  $\mathbf{x}(t)$ converges exponentially to  $\overline{\mathbf{x}}$.
\end{proof}
\begin{remark}
Since $\theta, \mu$ are related to the coupling in the players' cost functions  and $\lambda_2(L)$ to the connectivity between players, Theorem \ref{thm:LipPseudoComm} highlights the tradeoff between properties of the game and those of the communication graph $G_c$. Key is the fact that  the Laplacian contribution can be used to balance the other terms in  $\dot{V}$. Alternatively $\mathbf{L}$ on feedback path in Figure \ref{fig:expFeedback}  has excess passivity which compensates the lack of passivity in the $\mathbf{F}$ terms, or $\widetilde{\Sigma}$ on the forward path . 

\end{remark}
\begin{remark}
We note that we can relax the monotonicity assumption to hold just at the NE $x^*$, recovering a strict-diagonal assumption used in \cite{FKB12}. However, since $x^*$ is unknown, such an assumption cannot be checked a-priori except for special cases such as quadratic games, (see Section \ref{Simulation}).  Local results follow if assumptions for $F(\cdot)$ hold only locally around $x^*$, and for $\mathbf{F}(\cdot)$ only locally around $\mathbf{x}^*=\mathbf{1}_N \otimes x^*$ . 
We note that the class of quadratic games satisfies Assumption \ref{asmp:PseudoGradMono}(ii) globally.
\end{remark}

An alternative representation of the dynamics $\widetilde{\mathcal{P}}$,   (\ref{eq:overallDyn}), reveals interesting connections to distributed optimization and passivity based control,  \cite{WE11}, \cite{A07}.  To do that we use two matrices to write a compact representation for   $\widetilde{\mathcal{P}}$,   (\ref{eq:overallDyn}). Let   $\mathcal{S} = diag(\mathcal{S}_1, ... , \mathcal{S}_N) \in \reals^{(Nn - n)\times Nn}$, where    $\mathcal{S}_i$ in (\ref{eq:actualSMatrix}). Then  $\mathcal{S} $ and $\mathcal{R}$ (\ref{eq:actualStratREMatrix}) satisfy $\mathcal{S}\mathcal{R}^T = \mathbf{0}$ and 
\begin{align}
\mathcal{R}^T\mathcal{R} + \mathcal{S}^T\mathcal{S} = I, \,\,\mathcal{R}\mathcal{R}^T = I, \,\,	\mathcal{R}\mathcal{S}^T = \mathbf{0},
\,\,	\mathcal{S}\mathcal{S}^T = I 	\label{eq:IdRSMatrix} 
\end{align}
Using $ \mathcal{R}$ and $ \mathcal{S}$,  the stacked actions are  $ [(\mathbf{x}^1_1)^T,\dotsc,(\mathbf{x}^N_N)^T]^T = \mathcal{R}\mathbf{x}$, while the stacked estimates $ [(\mathbf{x}^1_{-1})^T,\dotsc,(\mathbf{x}^N_{-N})^T]^T = \mathcal{S}\mathbf{x}$. Let  $x=\mathcal{R}\mathbf{x}$ and $z =\mathcal{S}\mathbf{x}$, and using properties of $ \mathcal{R}$, $ \mathcal{S}$, (\ref{eq:IdRSMatrix}), yields $\mathbf{x} = \mathcal{R}^Tx + \mathcal{S}^Tz$. Thus  the equivalent representation of  $\widetilde{\mathcal{P}}$ (\ref{eq:overallDyn}) is:
\begin{align} \label{eq:actEstSplit1}
	\widetilde{\mathcal{P}} : \begin{cases}
	\dot{x} &= - \mathbf{F}(\mathcal{R}^Tx + \mathcal{S}^Tz) - \mathcal{R} \mathbf{L}[\mathcal{R}^Tx + \mathcal{S}^Tz] \\
	\dot{z} &= - \mathcal{S}\mathbf{L}[\mathcal{R}^Tx + \mathcal{S}^Tz] 
	\end{cases}
\end{align}
which separates the actions and the estimates stacked components of the dynamics. Using $\mathbf{L} = L \otimes I_n$ and $L=QQ^T$, with $Q$ the incidence matrix, yields the interconnected block-diagram in Figure \ref{fig:expFeedback1}. 

We note that, unlike  distributed optimization (e.g. \cite{WE11}) and passivity based control (e.g. \cite{A07}) where the dynamics are decoupled, in Figure \ref{fig:expFeedback}  the dynamics  on the top path are coupled, due to the inherent coupling in players' cost functions in  game  $\mathcal{G}(\mathcal{I},J_i, \Omega_i)$. As another observation, recall that given a matrix $Q$, pre-multiplication by $Q$ and post-multiplication by $Q^T$ preserves passivity of a system. Figure  \ref{fig:expFeedback} shows that possible generalized dynamics can be designed by substituting the identity block on the feedback path by some other passive dynamics.  Based on Figure \ref{fig:expFeedback}, one such dynamic generalization can be obtained by substituting the static feedback through $\mathbf{L}$, with an integrator or proportional-integrator term through $ \mathbf{L}$, which preserves passivity, as in \cite{WE11}. Thus the passivity interpretation of the game NE seeking dynamics design allows a systematic derivation of new dynamics/algorithms.

Note that  if the dynamics of the estimates were \emph{modified} such that the system approached quickly the consensus subspace then convergence to the  NE could be shown  via a time-scale decomposition approach. This is explored in the next section.

\begin{figure}[h!]
  \centering
\begin{tikzpicture}[auto, node distance=0.7cm]
	\node [block] (R) {$-\mathcal{R}$};
    \node [block, right = of R] (dynamics) {$	\dot{x} = -\mathbf{F}(x,z) + u$}; 
	\node [block, right = of dynamics] (RT) {$ \mathcal{R^T}$};

    \node [block, below = of dynamics] (int) {$\begin{array}{ccc} \frac{1}{s} & & \mathbf{0} \\ & \ddots & \\ \mathbf{0} & & \frac{1}{s} \end{array}$};
	\node [block, left = of int] (S) {$-\mathcal{S}$};
	\node [input, draw=none, left = of S] (start) {};
	\node [block, right = of int] (ST) {$ \mathcal{S^T}$};
	\node [sum, right = of ST] (sum) {+};
	
	\node [block, below = of int] (id) {$Id$};
	\node [block, left = of id] (Q) {$ Q \otimes I_n$};
	\node [block, right = of id] (QT) {$ Q^T \otimes I_n$};
	    
    \draw [->] (dynamics) -- node [name=x] {$x$}(RT);
    \draw [->] (RT) -| node {} (sum);
    \draw [->] (ST) -- node {} (sum);
    \draw [->] (int) -- node {$z$} (ST);
    \draw [->] (S) -- node {$\dot{z}$} (int);
    \draw [->] (QT) -- node {} (id);
    \draw [->] (id) -- node {} (Q);
    \draw [-] (Q) -| node {} (start);
    \draw [->] (start) |- node {}(R);
    \draw [->] (start) |- node {}(S);
    \draw [->] (R) -- node {$u$} (dynamics);
    \draw [->] (sum) |- node[pos=0.8] {$\mathbf{x}$} (QT);
    \draw [->] (int) -- node {$z$} (dynamics);
\end{tikzpicture}
	\caption{Block Diagram of Actions and Estimates dynamics in (\ref{eq:overallDyn})}
	\label{fig:expFeedback1}
\end{figure}
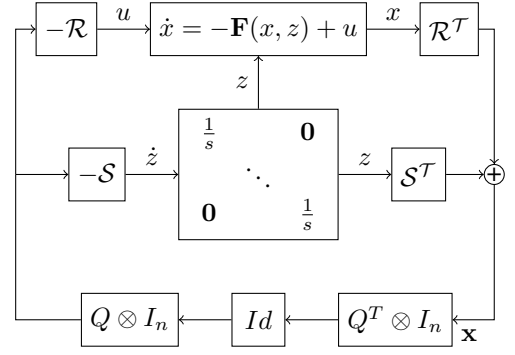

\vspace{-0.2cm}

\subsection{Two-Timescale Singular Perturbation Analysis}\label{sec:twoTime}

In this section we relax the connectivity bound on $L$ in Theorem \ref{thm:LipPseudoComm}, based on a time-scale separation argument.  The idea is to modify  the dynamics of the estimates in $\widetilde{\mathcal{P}}$ (\ref{eq:overallDyn}) or (\ref{eq:actEstSplit1}) such that the system approaches quickly the consensus subspace. Under Assumption \ref{asmp:PseudoGradMono}(ii) and  \ref{asmp:strongPseudo}(ii), we show  convergence to the  NE over \emph{a sufficiently connected}  $G_c$  based on a time-scale decomposition approach. 

Recall   the equivalent representation of  $\widetilde{\mathcal{P}}$ in (\ref{eq:actEstSplit1}) and modify the estimate component of the dynamics such that is much faster than the action component, 
\begin{align} \label{eq:actEstSplit}
	\widetilde{\mathcal{P}}_\epsilon : \begin{cases}
	\dot{x} &= - \mathbf{F}(\mathcal{R}^Tx + \mathcal{S}^Tz) - \mathcal{R} \mathbf{L}[\mathcal{R}^Tx + \mathcal{S}^Tz] \\
	\epsilon\dot{z} &= -\mathcal{S}\mathbf{L}[\mathcal{R}^Tx + \mathcal{S}^Tz]
	\end{cases}
\end{align}
where  $\epsilon >0$. 
Thus player $i$'s dynamics is as follows:
\begin{align} \label{eq:agentActionEstimateEPS}
	\widetilde{\mathcal{P}}_{i,\epsilon} : \begin{cases}
	\begin{bmatrix}
	\dot{x}_i \\
	\dot{\mathbf{x}}^i_{-i}
	\end{bmatrix} =
	\begin{bmatrix}
	-\nabla_i J_i(x_i,\mathbf{x}^i_{-i}) - \mathcal{R}_i\sum_{j\in\mathcal{N}_i} ( \mathbf{x}^i - \mathbf{x}^j) \\
	- \frac{1}{\epsilon}\mathcal{S}_i(\sum_{j\in\mathcal{N}_i} \mathbf{x}^i - \mathbf{x}^j)\\
	\end{bmatrix}
	\end{cases}
\end{align}
with the $\frac{1}{\epsilon}$ high gain on the estimate component. $\widetilde{\mathcal{P}}_\epsilon$ (\ref{eq:actEstSplit})  is in the standard form of a singularly perturbed system, where the estimate dynamics and the action dynamics are the \emph{fast} and the \emph{slow} components, respectively. 
\begin{thm} \label{thm:STRGMonPseudoComm}
Consider a game $\mathcal{G}(\mathcal{I},J_i, \Omega_i)$ over a communication graph $G_c$ under Assumptions  \ref{asmp:CxnGraph}, \ref{asmp:Jsmooth}(i), \ref{asmp:PseudoGradMono}(ii) and  \ref{asmp:strongPseudo}(ii).  Let each player's dynamics   $\widetilde{\mathcal{P}}_{i,\epsilon}$ be as in (\ref{eq:agentActionEstimateEPS}), or overall   $\widetilde{\mathcal{P}}_\epsilon$   (\ref{eq:actEstSplit}), $\epsilon >0$.  Then, there exists $\epsilon^* >0$, such that for all $0<\epsilon <\epsilon^*$,  $(x^*, \mathcal{S} (\mathbf{1}_N\otimes x^*))$ is exponentially stable. Alternatively,  $(x^*, \mathcal{S} (\mathbf{1}_N\otimes x^*))$ is asymptotically stable, for all $0<\epsilon<1$ such that
$$ \lambda_2(\mathbf{L})  > \epsilon   \sqrt{N}  ( \frac{\theta}{\mu}  +1) (\theta + 2 d^*)).  $$
\end{thm}


\begin{proof} We analyze  (\ref{eq:actEstSplit})  by examining the reduced and the boundary-layer systems. First we find the roots of $\mathcal{S}\mathbf{L}[\mathcal{R}^Tx + \mathcal{S}^Tz]=0$, or $\mathcal{S} \mathbf{L}\mathbf{x}=0$. Note that, by (\ref{eq:IdRSMatrix}),	$\mathbf{x} \in Null(\mathcal{S} \mathbf{L})$ if and only if $ \mathbf{L}\mathbf{x}\in Null(\mathcal{S})$, which is equivalent to $\mathbf{L}\mathbf{x}\in Range(\mathcal{R^T})$. Thus $\mathbf{x} \in Null(\mathcal{S}  \mathbf{L})$ if and only if there exists $q \in \reals^n$ such that $\mathbf{L}\mathbf{x} = \mathcal{R^T}q$. Then for such $ q \in \reals^n$ and for all $w \in \reals^n$, $(\mathbf{1}^T_N \otimes w^T)\mathbf{L}\mathbf{x} = (\mathbf{1}^T_N \otimes w^T)\mathcal{R}^T q$. Using $(\mathbf{1}^T_N \otimes w^T)\mathbf{L} \mathbf{x} =0$ by (\ref{eq:LIeqZero}) and $\mathcal{R}$ in (\ref{eq:actualStratREMatrix}), this means $0=w^T q$, for all $w \in \reals^n$. Therefore, $q=0$ and $\mathbf{L}\mathbf{x} =0$. By (\ref{eq:nullExpL}), $\mathbf{x}=\mathbf{1}_N\otimes x$, $x\in \Omega$. Hence roots of  $\mathcal{S}\mathbf{L}[\mathcal{R}^Tx + \mathcal{S}^Tz]=0$ are when $\mathbf{x} =(\mathbf{1}_N\otimes x)$, i.e.,  $z= \mathcal{S} (\mathbf{1}_N\otimes x)$.

We use a change of coordinates $v=z-\mathcal{S} (\mathbf{1}_N\otimes x)$,  to shift the equilibrium of the boundary-layer system to the origin. First, we use  $z=v+\mathcal{S} (\mathbf{1}_N\otimes x)$ and $x = \mathcal{R}(\mathbf{1}_N\otimes x)$ to rewrite the term $\mathcal{R}^Tx + \mathcal{S}^Tz$ that appears in (\ref{eq:actEstSplit}) as follows,
\begin{align*}
& \mathcal{R}^Tx + \mathcal{S}^Tz	= \mathcal{R}^T\mathcal{R}(\mathbf{1}_N\otimes x) + \mathcal{S}^T(v+\mathcal{S} (\mathbf{1}_N\otimes x )) \\
		&= (\mathcal{R}^T\mathcal{R} + \mathcal{S}^T\mathcal{S})(\mathbf{1}_N\otimes x) + \mathcal{S}^T\, v = \mathbf{1}_N\otimes x + \mathcal{S}^T\, v 
\end{align*}
where  (\ref{eq:IdRSMatrix}) was used. Using this and the  change of variables $v=z-\mathcal{S} (\mathbf{1}_N\otimes x)$ into (\ref{eq:actEstSplit}) with $ \mathbf{L}(\mathbf{1}_N\otimes x)=0$ yields,
\begin{align} \label{eq:actEstSplitxv2}
	\widetilde{\mathcal{P}}_\epsilon : \begin{cases}
	\dot{x} &= - \mathbf{F}(\mathbf{1}_N\otimes x + \mathcal{S}^T\, v) - \mathcal{R} \mathbf{L} \mathcal{S}^T\, v  \\
	\epsilon \dot{v} &= -\mathcal{S}\mathbf{L} \mathcal{S}^T\, v +  \epsilon \mathcal{S} (\mathbf{1}_N\otimes  \mathcal{R} \mathbf{L} \mathcal{S}^T\, v) \\
	& +  \epsilon \mathcal{S} (\mathbf{1}_N\otimes \mathbf{F}(\mathbf{1}_N\otimes x + \mathcal{S}^T\, v) )
	\end{cases}
\end{align}
Note that $v=0$ is the quasi-steady state of $\epsilon \dot{v}$ and substituting this in $\dot{x}$ gives  the reduced system as
\begin{align} \label{eq:actEstSplitxv2RED}
	\dot{x} = - \mathbf{F}(\mathbf{1}_N\otimes x) = -F(x)
\end{align}
which is exactly the gradient dynamics and has equilibrium $x^*$at the NE. By Lemma \ref{lemma:perfectInfo}, under Assumption \ref{asmp:PseudoGradMono}(ii) the gradient dynamics,  (\ref{eq:actEstSplitxv2RED}), is exponentially stable. 

 The boundary-layer system on the $\tau=t/\epsilon$ timescale is
\begin{align} \label{eq:actEstSplitxv2BDRY}
	\frac{d v}{d\tau}  = -\mathcal{S}\mathbf{L} \mathcal{S}^T\, v 
\end{align}
It can be shown that matrix $\mathcal{S}\mathbf{L}\mathcal{S}^T$ is   positive definite, so that (\ref{eq:actEstSplitxv2BDRY}) is exponentially stable. 
To see this note that $\mathcal{S}\mathbf{L}\mathcal{S}^Tv= 0$ only if $\mathcal{S}^Tv \in Null(\mathcal{S}\mathbf{L})$.  Recall that  $Null(\mathcal{S}\mathbf{L}) = Null(\mathbf{L}) =(\mathbf{1}_N \otimes w)$, $w\in {\Omega}$. Note that to be in $Null(\mathbf{L})$, $y = [(y^1)^T\dotsc (y^N)^T]^T = \mathcal{S}^Tv$ has to have $y^i = y^j$, $\forall i,j\in\mathcal{I}$,  but $y$ has a $0$ in component $y^i_i\ \forall i\in\mathcal{I}$, due to the definition of $\mathcal{S}$. Therefore $y^i \neq y^j$ unless $y^j_i = y^i_i = 0$, for all $j$. Therefore $y=\mathcal{S}^Tv$ has to be equal to $0$ to be in the $Null(\mathbf{L})$.  Since  $Null(\mathcal{S}^T) = \{0\}$ this implies $v=0$, hence $Null(\mathcal{S}\mathbf{L}\mathcal{S}^T) = 0$ and  $\mathcal{S}\mathbf{L}\mathcal{S}^T$ is positive definite. 
  By Theorem 11.4 in \cite{K02} it follows that there exists $\epsilon^* >0$, such that for all $\epsilon <\epsilon^*$, $(x^*, 0)$ is exponentially stable for (\ref{eq:actEstSplitxv2}), or $(x^*, \mathcal{S} (\mathbf{1}_N\otimes x^*))$ is exponentially stable for (\ref{eq:actEstSplit}). 
  
 Alternatively,  Theorem 11.3 in \cite{K02} can be applied to (\ref{eq:actEstSplitxv2}) to show asymptotic stability. 
The two Lyapunov functions are $V(x) = \frac{1}{2}\| x- x^*\|^2$ and $W(v) = \frac{1}{2}\| v\|^2$, and along the reduced and the boundary layer-systems, (\ref{eq:actEstSplitxv2RED}),(\ref{eq:actEstSplitxv2BDRY}),  the following hold $$
-(x-x^*)^T  \mathbf{F}(\mathbf{1}_N\otimes x) \leq - \mu  \|x-x^*\|^2
$$
$$
-v^T\mathcal{S}\mathbf{L}\mathcal{S}^T v \leq - \lambda_2(L)/N \|v\|^2
$$
so that (11.39) and (11.40) in  \cite{K02} hold for $\alpha_1 = \mu$, $\psi_1(x) = \|x-x^*\|$, and $\alpha_2 = \lambda_2(L)/N$, $\psi_2(v) = \|v\|$.

Note also that  the following holds 
\begin{align*}
& (x-x^*)^T \left [ - \mathbf{F}(\mathbf{1}_N\otimes x + \mathcal{S}^T\, v) - \mathcal{R} \mathbf{L} \mathcal{S}^T\, v + \mathbf{F}(\mathbf{1}_N\otimes x)\right ] \\
& \leq (\theta +\lambda_N(L)) \|x-x^*\| \|v\|
\end{align*}
so that (11.43) in \cite{K02}  holds for $\beta_1 = \theta +\lambda_N(L) $.

Similarly,
\begin{align*}
& v^T \mathcal{S} (\mathbf{1}_N\otimes  \mathcal{R} \mathbf{L} \mathcal{S}^T\, v) 
 + v^T \mathcal{S} (\mathbf{1}_N\otimes \mathbf{F}(\mathbf{1}_N\otimes x + \mathcal{S}^T\, v) )\\
&\leq \sqrt{N} \theta \|x-x^*\| \|v\| + \sqrt{N} (\theta + \lambda_N(L) )  \| v\|^2
\end{align*}
and (11.44),  \cite{K02}  holds for $\beta_2  = \sqrt{N} \theta $, $\gamma = \sqrt{N} (\theta + \lambda_N(L)) $.

Then using Theorem 11.3, $\epsilon^* = \frac{\alpha_1\alpha_2}{\alpha_1 \gamma + \beta_1 \beta_2}$, is given as
$$
\epsilon^* = \frac{  \lambda_2(L) \mu }{ N \sqrt{N}  ( \theta +\mu ) (\theta + \lambda_N(L))  }
$$
and using $\lambda_N(L) \leq 2 d^* =2 \max_{i\in \mathcal{I}} | \mathcal{N}_i |$, 
$$
\epsilon^* \geq  \frac{  \lambda_2(L) \mu }{ N \sqrt{N}  ( \theta +\mu ) (\theta + 2 d^*))  }
$$
Then, by Theorem 11.3 in \cite{K02}, for any $0<\epsilon <1$ such that
$$
 \lambda_2(L)  > \epsilon N  \sqrt{N}  ( \frac{\theta}{\mu}  +1) (\theta + 2 d^*))  
$$
$(x^*, 0)$  is  asymptotic stable  for (\ref{eq:actEstSplitxv2}), hence $(x^*, \mathcal{S} (\mathbf{1}_N\otimes x^*))$ is  asymptotic stable for (\ref{eq:actEstSplit}).

 \end{proof}

\begin{remark}
In  $\widetilde{\mathcal{P}}_{i,\epsilon}$ (\ref{eq:agentActionEstimateEPS}) the  estimate dynamics is made faster with the  gain $1/\epsilon$. 
It can be shown  that for sufficiently high $1/\epsilon$, $\frac{1}{\epsilon} > N \sqrt{N}  (1 + 2 \frac{d^*}{\theta})  $, the bound on $\lambda_2(L)$  in  Theorem \ref{thm:STRGMonPseudoComm}  is lower than the bound on $\lambda_2(L)$ in Theorem \ref{thm:LipPseudoComm}. 
Alternatively, we can consider a  gain parameter $\frac{1}{\epsilon}>0$ on the estimates in (\ref{eq:overallDyn}) to improve the lower bound to $\lambda_2(L)  > \epsilon(\frac{\theta^2}{\mu} + \theta)$, as shown in the next section.  Thus a higher $1/\epsilon $ can relax the connectivity bound on $L$, but $\epsilon$ is a global parameter.
 This highlights another aspect of the tradeoff between game properties (coupling), communication graph (consensus) properties and information. 
\end{remark}

\section{Projected NE Dynamics for Compact Action Sets}\label{projection}

In this section we treat the case of compact $\Omega_i$ action sets, under Assumption \ref{asmp:Jsmooth}(ii), using projected dynamics. We highlight the major steps of the approach and their differences compared to the unconstrained action set case.

In a game of perfect information, for  compact $\Omega_i$ action set, under Assumption \ref{asmp:Jsmooth}(ii) each player $i \in \mathcal{I}$ runs the projected gradient-based dynamics,  given as \cite{Flam02}, \cite{SA05},
\begin{align} \label{eq:gradientProj}
	\mathcal{P}_i : & \quad	\dot{x}_i = \Pi_{\Omega_i}(x_i,-\nabla_i J_i(x_i, x_{-i})), \, \, x_i(0) \in \Omega_i
\end{align}
The overall  system of all agents' projected dynamics in stacked notation is given by
\begin{align} \label{eq:gradientProj_all}
\mathcal{P} : \,\,\, \dot{x}(t) = \Pi_{\Omega}\left (x(t),-F(x(t) )\right ), \quad x(0) \in \Omega
\end{align}
or, equivalently, 
$$\mathcal{P} : \,\,\, \dot{x}(t) =P_{T_\Omega(x(t))}\left [-F(x(t) \right ], \quad x(0) \in \Omega$$
where equivalence follows by using Lemma 1 (Moreau's decomposition theorem, \cite{Lemarechal}), or directly by Proposition 1 and Corollary 1 in \cite{Brogliato2006}. Furthermore,  this  is equivalent to the differential inclusion  \cite{AubinCellina} 
\begin{align} \label{eq:gradientProj_all_DI}
 -F(x(t)) - \Pi_{\Omega}\left (x(t),-F(x(t) )\right ) \in N_{\Omega}(x(t)).
\end{align}
In all the above  the projection operator is discontinuous on the boundary of $\Omega$. We use the standard definition of a solution of a projected dynamical system (PDS), (Definition 2.5 in \cite{NZ96}). Thus we call $x:[0,+\infty) \rightarrow \Omega$ a solution of (\ref{eq:gradientProj_all}) if $x(\cdot)$ is an absolutely continuous function $t \mapsto x(t)$ and $\dot{x}(t) = \Pi_{\Omega}\left (x(t),-F(x(t) )\right )$ holds almost everywhere (a.e.)  with respect to $ t$, i.e., except on a set of measure zero. 

The existence of a unique  solution of  (\ref{eq:gradientProj_all}) is guaranteed for any $x(0) \in \Omega$, under Lipschitz continuity of $F$ on $\Omega$, cf. Theorem 2.5 in \cite{NZ96}. Note that any solution must necessarily lie in $\Omega$ for almost every $t$. 
Alternatively, existence holds under continuity and \emph{(hypo) monotonicity} of $F$, i.e., for some $\mu \leq  0$, 
$$
(x-x')^T(F(x) - F(x')) \geq  \mu \| x-x'\|^2,  \quad \forall x,x' \in \Omega
$$
(see Assumption 2.1 in  \cite{NZ96},  and also Theorem 1 in  \cite{Brogliato2006} for extension to a non-autonomous systems). 
This is similar to the QUAD relaxation in   \cite{DDR11}. It means that $F(x) - \mu x $ is monotone, where $\mu \leq  0$. Note that  $F(x) + \eta x $ is  strongly monotone for any $\eta >  - \mu$, with $\eta +\mu >0$ monotonicity constant. When $\mu > 0$ in fact we can take $\eta=0$ and recover $\mu$-strong monotonicity of $F$. Thus under Assumption \ref{asmp:Jsmooth}(ii), \ref{asmp:PseudoGradMono}(i), for any $x(0) \in \Omega$, there exists a unique solution of   (\ref{eq:gradientProj_all}) and moreover, a.e., $\dot{x}(t) \in T_{\Omega}(x(t))$ and  $x(t) \in \Omega$.


Equilibrium points of (\ref{eq:gradientProj_all}) coincide  with Nash equilibria, which are solutions of the VI($F$, $\Omega$), by Theorem 2.4 in \cite{NZ96}. To see this, let $\overline{x}$ be an equilibrium point of (\ref{eq:gradientProj_all}) such that $\{ \overline{x} \in \Omega \, | \, \mathbf{0}_{n} = \Pi_{\Omega}\left (\overline{x},-F(\overline{x})\right )\}$. 
By Lemma 2.1 in \cite{NZ96} and (\ref{charactPi}), if $\overline{x} \in \text{int} \Omega$, then $\mathbf{0}_{n} = \Pi_{\Omega}\left (\overline{x},-F(\overline{x})\right )=-F(\overline{x}) $, while if  $\overline{x} \in \partial \Omega$, then 
\begin{align} \label{eq:NE_proj_with_n}
\mathbf{0}_{n} = \Pi_{\Omega}\left (\overline{x},-F(\overline{x})\right ) = - F(\overline{x}) - \beta \, n
\end{align}
for some $\beta>0$ and $n \in n(\overline{x}) \subset N_\Omega(\overline{x})$. Equivalently, by (\ref{eq:gradientProj_all_DI}),  $-F(\overline{x}) \in N_\Omega(\overline{x})$  
  and  using the definition of $N_\Omega(\overline{x})$, it follows that 
$$
-F(\overline{x})^T \, (x - \overline{x}) \leq 0, \quad \forall x\in \Omega
$$
Comparing to (\ref{eq:ViNash}), or (\ref{eq:ViNashNormCone}), it follows that $\overline{x} = x^*$. Thus the equilibrium points of (\ref{eq:gradientProj_all})  	$\{x^* \in \Omega \, | \, \mathbf{0}_{n}= \Pi_\Omega (x^*, - F(x^*)) \}$ coincide  with Nash equilibria $x^*$. 


\begin{lemma} \label{lemma:perfectInfoProj}
	Consider a game $\mathcal{G}(\mathcal{I},J_i, \Omega_i)$ in the perfect information case, under Assumptions \ref{asmp:Jsmooth}(ii) and \ref{asmp:PseudoGradMono}(i).  Then, for any $x_i(0) \in \Omega_i$, the solution  of  (\ref{eq:gradientProj}), $i \in \mathcal{I}$, or (\ref{eq:gradientProj_all}) converges asymptotically to the NE of the game $x^*$. Under Assumption \ref{asmp:PseudoGradMono}(ii) convergence is exponential. 
\end{lemma}
\begin{proof} 
The proof follows from Theorem 3.6 and 3.7 in  \cite{NZ96}. Consider  any $x(0) \in \Omega$  and $V(t,x) = \frac{1}{2} \| x(t) - x^* \|^2$, where $x^*$ is the Nash equilibrium of the game, and $x(t)$ is the solution of   (\ref{eq:gradientProj_all}). Then the time derivative of $V$ along solutions of (\ref{eq:gradientProj_all}) is $\dot{V} = (x(t)-x^*)^T \Pi_{\Omega}(x(t),-F(x(t)))$. Since $T_\Omega (x(t)) = [N_\Omega(x(t))]^o$, by Moreau's decomposition theorem (Lemma \ref{lemma:Moreau}), at any point $x(t)\in \Omega$ the pseudo-gradient  $-F(x(t))$ can be decomposed  as in  (\ref{v_Moreau}) into  normal and tangent components, in  $N_{\Omega}(x(t))$ and  $T_{\Omega}(x(t))$. Since  $\Pi_{\Omega}(x(t),-F(x(t))) = P_{T_\Omega(x(t))}(-F(x(t)))$ is  in the tangent cone $T_{\Omega}(x(t))$, it follows as in (\ref{eq:gradientProj_all_DI}) that 
	\begin{align}\label{eq:FinNcone}
		-F(x(t))-\Pi_{\Omega}(x(t),-F(x(t))) \in N_{\Omega}(x(t)) 
	\end{align}
	From the definition of the normal cone $N_{\Omega}(x(t))$ this means,
	\begin{align*}
		(x'-x(t))^T (-F(x(t))-\Pi_{\Omega}(x(t),-F(x(t)))) &\leq 0 
	\end{align*}	
for all $x'\in {\Omega}$. Thus  it follows that for $x'=x^*$ and $\forall x(t) \in \Omega$, 
$$	(x(t)-x^*)^T	\Pi_{\Omega}(x(t),-F(x(t))))  \leq -(x(t)-x^*)^T  F(x(t)) 
$$
	From (\ref{eq:ViNash}),  at the Nash equilibrium $F(x^*)^T(x(t)-x^*) \geq 0$, $\forall x(t) \in \Omega$. Therefore adding this to the right-hand side of the above and using $\dot{V}(t)$ yields that 
along solutions of (\ref{eq:gradientProj_all}), for all $t \geq 0$, $ \dot{V} = (x(t)-x^*)^T \Pi_{\Omega}(x(t),-F(x(t))) \leq -(x(t)-x^*)^T (F(x(t)) -   F(x^*)) < 0$,   when $x(t) \neq x^*$,  
where the strict inequality follows from Assumption \ref{asmp:PseudoGradMono}(i). Hence $V(t)$ is monotonically decreasing and non-negative, and thus there exists $\lim_{t\rightarrow \infty} V(t) = \underline{V} $. As in Theorem 3.6 in  \cite{NZ96}, a contradiction argument can be used to show that $\underline{V} = 0$, hence  for any $x(0) \in \Omega$, $\| x(t) - x^* \| \rightarrow 0$ as $t \rightarrow \infty$.


Under Assumption \ref{asmp:PseudoGradMono}(ii), for any $x(0) \in \Omega$, along solutions of (\ref{eq:gradientProj_all}), for all $t \geq 0$, $\dot{V} \leq - \mu  \|x(t) - x^* \|^2 = -  \mu V(t)$, $\mu >0$, $\forall x$ and exponential convergence follows immediately.
\end{proof}

In the partial or networked information case, over graph $G_c$, we modify each player's 	$\widetilde{\Sigma}_i$, in (\ref{eq:baseAgentExpFiner}) using projected dynamics for the action components to $\Omega_i$, as in 
\begin{align}
	\label{eq:baseAgentExpFinerProj}
	\widetilde{\Sigma}_i &: \begin{cases}
		\begin{bmatrix}
			\dot{x}_i \\
			\dot{\mathbf{x}}^i_{-i}
		\end{bmatrix}
		= \begin{bmatrix}
			\Pi_{\Omega_i} \left ( x_i, -\nabla_i J_i(x_i, \mathbf{x}^i_{-i})+B^i_i\mathbf{u}_i(t) \right )  \\
			 B^i_{-i} \mathbf{u}_i
		\end{bmatrix}  \\
		\mathbf{y}_i = (B^i)^T\mathbf{x}^i
	\end{cases}
\end{align}
where   $\mathbf{u}_i(t) \in \reals^{n}$ is a piecewise continuous function, to be designed based on the relative output feedback from its neighbours,  such that $ \mathbf{x}^i = \mathbf{x}^j$, for all $i,j$, and converge towards the NE $x^*$. Write $\widetilde{\Sigma}_i$ (\ref{eq:baseAgentExpFinerProj}) in a more compact form
\begin{align}
 \label{eq:baseAgentExpProj} 
	\widetilde{\Sigma}_i &: \begin{cases}
		\dot{\mathbf{x}}^i = \mathcal{R}_i^T\Pi_{\Omega_i}\left (x_i,-\nabla_i J_i(\mathbf{x}^i)+\mathcal{R}_iB^i\mathbf{u}_i\right ) + \mathcal{S}_i^T\mathcal{S}_iB^i\mathbf{u}_i \\
		\mathbf{y}_i = (B^i)^T\mathbf{x}^i
	\end{cases}
\end{align}
where $\mathcal{R}_i$, $\mathcal{S}_i $, are defined as in (\ref{eq:actualStratREMatrix}), (\ref{eq:actualSMatrix}). The overall dynamics  for all players becomes in stacked form, $\widetilde{\Sigma}$,
 \begin{align}\label{eq:baseAgentProj}
	\widetilde{\Sigma} : \begin{cases}
		\dot{\mathbf{x}} = \mathcal{R}^T\Pi_{\Omega}\left (\mathcal{R}\mathbf{x}(t),-\mathbf{F}(\mathbf{x}(t))+\mathcal{R}B\mathbf{u}(t) \right ) + \mathcal{S}^T\mathcal{S}B\mathbf{u}(t) \\
		\mathbf{y} = B^T\mathbf{x}(t)
	\end{cases}
\end{align}
where $x= \mathcal{R}\mathbf{x}$, $\mathcal{R} = diag(\mathcal{R}_1, \dots, \mathcal{R}_N)$, $\mathcal{S} = diag(\mathcal{S}_1, \dots, \mathcal{S}_N)$, satisfying the properties in (\ref{eq:IdRSMatrix}), and $\mathbf{u}(t) \in \reals^{Nn}$ is piecewise continuous. This is similar to (\ref{eq:baseAgent}), except that  the dynamics for the action components is projected to $\Omega$. For any $\mathcal{R}\mathbf{x}(0)\in {\Omega}$,  existence of a unique solution of (\ref{eq:baseAgentProj}) is guaranteed under Assumption \ref{asmp:strongPseudo}(i) or (ii) by Theorem 1 in  \cite{Brogliato2006}. Extending the incrementally passivity and EIP concept to projected dynamical systems leads to the following result. 

\begin{lemma}  \label{lemma:ForwardEIPProj}
	Under Assumption \ref{asmp:strongPseudo}(i), the overall system $\widetilde{\Sigma}$, (\ref{eq:baseAgentProj}), is incrementally passive, hence EIP.
\end{lemma}
\begin{proof}
Consider two inputs  $\mathbf{u}(t)$, $\mathbf{u}'(t)$ and  let $\mathbf{x}(t)$, $\mathbf{x}'(t)$, $\mathbf{y}(t)$, $\mathbf{y}'(t)$ be the state trajectories and outputs of $\widetilde{\Sigma}$ (\ref{eq:baseAgentProj}). Let the storage function be $V(t,\mathbf{x},\mathbf{x}') = \frac{1}{2}\|\mathbf{x}(t)-\mathbf{x}'(t)\|^2$. Then, along solutions of (\ref{eq:baseAgentProj}),
\begin{align}\label{eq:dotVagain}
\dot{V} & =(\mathbf{x}(t)-\mathbf{x}'(t))^T\mathcal{R}^T[\Pi_{\Omega}(\mathcal{R}\mathbf{x}(t),-\mathbf{F}(\mathbf{x}(t))+\mathcal{R}B\mathbf{u}(t)) \notag\\
	&- \Pi_{\Omega}(\mathcal{R}\mathbf{x'}(t),-\mathbf{F}(\mathbf{x'}(t))+\mathcal{R}B\mathbf{u}'(t))]  \\
	& + (\mathbf{x}(t)-\mathbf{x}'(t))^T\mathcal{S}^T\mathcal{S}B(\mathbf{u}(t) - \mathbf{u}'(t)) \notag
\end{align}
Notice that using  (\ref{eq:IdRSMatrix}) the following holds for (\ref{eq:baseAgentProj}),
$$
	\mathcal{R}\dot{\mathbf{x}}(t) = \Pi_{\Omega}\left (\mathcal{R}\mathbf{x}(t),-\mathbf{F}(\mathbf{x}(t))+\mathcal{R}B\mathbf{u}(t) \right ) \in  T_ {\Omega}( \mathcal{R}\mathbf{x}(t))
$$
Hence as in (\ref{eq:FinNcone}), for any  $\mathcal{R}\mathbf{x}(t)\in {\Omega}$,
\begin{align*}
	-\mathbf{F}(\mathbf{x})(t) + \mathcal{R}B\mathbf{u}(t)  -\Pi_{\Omega}\left (\mathcal{R}\mathbf{x}(t), -\mathbf{F}(\mathbf{x}(t))+\mathcal{R}B\mathbf{u}(t) \right ) \\
	 \in N_{\Omega}(\mathcal{R}\mathbf{x}(t)) 
\end{align*}
and using the definition of normal cone $N_{\Omega}(\mathcal{R}\mathbf{x}(t))$,
\begin{align}
	(\mathcal{R}\mathbf{x}(t)-\mathcal{R}\mathbf{x'}(t))^T \Pi_{\Omega} \left (\mathcal{R}\mathbf{x}(t),-\mathbf{F}(\mathbf{x}(t))+\mathcal{R}B\mathbf{u}(t) \right ) \leq 
	\notag \\
	(\mathcal{R}\mathbf{x}(t)-\mathcal{R}\mathbf{x'}(t))^T (-\mathbf{F}(\mathbf{x}(t)) + \mathcal{R}B\mathbf{u}(t)) \label{eq:eipProj1}
\end{align}
for all $\mathcal{R}\mathbf{x}'(t)\in {\Omega}$. Since  $\mathcal{R}\mathbf{x}(t)\in {\Omega}$ and $\mathcal{R}\mathbf{x}'(t)\in {\Omega}$ are both arbitrary elements in $\Omega$, swapping them leads to 
\begin{align}
	(\mathcal{R}\mathbf{x'}(t)-\mathcal{R}\mathbf{x}(t))^T\Pi_{\Omega}\left (\mathcal{R}\mathbf{x'}(t),-\mathbf{F}(\mathbf{x}'(t))+\mathcal{R}B\mathbf{u}'(t) \right ) \leq \notag \\
	(\mathcal{R}\mathbf{x'}(t)-\mathcal{R}\mathbf{x}(t))^T(-\mathbf{F}(\mathbf{x}'(t)) + \mathcal{R}B\mathbf{u}'(t)) \label{eq:eipProj2}
\end{align}
Adding (\ref{eq:eipProj1}) and (\ref{eq:eipProj2}) results in
\begin{align*}
	(\mathbf{x}(t)- &\mathbf{x'}(t))^T \mathcal{R}^T  [ \Pi_{\Omega}  (\mathcal{R}\mathbf{x}(t),-\mathbf{F}(\mathbf{x}(t))+\mathcal{R}B\mathbf{u}(t)  ) \notag \\
	&- \Pi_{\Omega}  (\mathcal{R}\mathbf{x'}(t),-\mathbf{F}(\mathbf{x}'(t))+\mathcal{R}B\mathbf{u}'(t)  )] \leq \notag \\
	&- (\mathbf{x}(t)-\mathbf{x'}(t))^T \mathcal{R}^T (\mathbf{F}(\mathbf{x}(t))-\mathbf{F}(\mathbf{x}'(t))) \notag \\
	&+(\mathbf{x}(t)-\mathbf{x}'(t))^T\mathcal{R}^T\mathcal{R}B(\mathbf{u}(t)-\mathbf{u}'(t))
\end{align*}
Therefore using this in (\ref{eq:dotVagain}), yields for $\dot{V}$ 
\begin{align*}
	&\dot{V}\leq - (\mathbf{x}(t)-\mathbf{x'}(t))^T \mathcal{R}^T(\mathbf{F}(\mathbf{x}(t))-\mathbf{F}(\mathbf{x}'(t))) \notag \\
	&\, +(\mathbf{x}(t)-\mathbf{x}'(t))^T\mathcal{R}^T\mathcal{R}B(\mathbf{u}(t)-\mathbf{u}'(t)) \\
	& \, + (\mathbf{x}(t)-\mathbf{x}'(t))^T\mathcal{S}^T\mathcal{S}B(\mathbf{u}(t) - \mathbf{u}'(t))
\end{align*}or,  using $\mathcal{R}^T\mathcal{R} + \mathcal{S}^T\mathcal{S} = I$,
\begin{align}\label{eq:EIP_projbis}
	\dot{V}
	&\leq -(\mathbf{x}(t)-\mathbf{x'}(t))^T \mathcal{R}^T(\mathbf{F}(\mathbf{x}(t))-\mathbf{F}(\mathbf{x}'(t))) \notag \\
	& + (\mathbf{x}(t)-\mathbf{x}'(t))^TB(\mathbf{u}(t) - \mathbf{u}'(t))
\end{align}
Finally, using  Assumption \ref{asmp:strongPseudo}(i), it follows that 
 \begin{align*}
	\dot{V} 	&\leq (\mathbf{y}(t)-\mathbf{y}'(t))^T(\mathbf{u}(t) - \mathbf{u}'(t))
\end{align*}
and $\widetilde{\Sigma}$ is incrementally passive, hence EIP.
\end{proof}
Since   $\widetilde{\Sigma}$, (\ref{eq:baseAgentProj}) is incrementally passive by Lemma \ref{lemma:ForwardEIPProj}, and $\mathbf{L}$ is positive semi-definite, as in Section \ref{commGradDyn} we consider a passivity-based control  $\mathbf{u}(t) = -\mathbf{L} \mathbf{x}(t)$. The resulting closed-loop system which represents the new overall system dynamics $\widetilde{\mathcal{P}}$ is given in stacked notation as
 \begin{align} \label{eq:overallDynProj}
	\widetilde{\mathcal{P}} :& \quad 
 		\dot{\mathbf{x}} = \mathcal{R}^T\Pi_{\Omega} \left (\mathcal{R}\mathbf{x}(t),-\mathbf{F}(\mathbf{x}(t))-\mathcal{R}\mathbf{L}\mathbf{x}(t) \right ) - \mathcal{S}^T\mathcal{S}\mathbf{L}\mathbf{x}(t)
\end{align}
Alternatively,  using  $x=\mathcal{R}\mathbf{x}$, $z= \mathcal{S}\mathbf{x}$, $\mathcal{R} \mathcal{S}^T=0$, equivalently with actions and estimates separated as in (\ref{eq:actEstSplit1}), 
\begin{align} \label{eq:actEstSplit1Proj}
	\widetilde{\mathcal{P}} : \begin{cases}
	\dot{x} &= \Pi_{\Omega} \left (x, - \mathbf{F}(\mathcal{R}^Tx + \mathcal{S}^Tz) - \mathcal{R} \mathbf{L}[\mathcal{R}^Tx + \mathcal{S}^Tz] \right )\\
	\dot{z} &= - \mathcal{S}\mathbf{L}[\mathcal{R}^Tx + \mathcal{S}^Tz] 
	\end{cases}
\end{align}
Existence of a unique solution of (\ref{eq:overallDynProj}) or (\ref{eq:actEstSplit1Proj}) is guaranteed under Assumption \ref{asmp:strongPseudo}(i) or (ii) by Theorem 1 in  \cite{Brogliato2006}. From  (\ref{eq:overallDynProj}) or (\ref{eq:actEstSplit1Proj}) 
after separating the action $\mathbf{x}^i_i=x_i$ and  estimate $\mathbf{x}^i_{-i}$  dynamics, the new projected player dynamics $\widetilde{\mathcal{P}}_i$ are, 
\begin{align} \label{eq:agentActionEstimateProj}
	\widetilde{\mathcal{P}}_i : \begin{cases}
	\dot{x}_i  & = \Pi_{\Omega_i} \left (x_i,-\nabla_i J_i(\mathbf{x}^i) - \mathcal{R}_i\sum_{j\in\mathcal{N}_i} ( \mathbf{x}^i - \mathbf{x}^j) \right ) \\
	\dot{\mathbf{x}}^i_{-i} &=	-\mathcal{S}_i(\sum_{j\in\mathcal{N}_i} \mathbf{x}^i - \mathbf{x}^j)
	\end{cases}
\end{align}
Compared to  (\ref{eq:agentActionEstimate}), $\widetilde{\mathcal{P}}_i$, in (\ref{eq:agentActionEstimateProj}) has projected action components.  The next result shows that the equilibrium of (\ref{eq:overallDynProj}) or (\ref{eq:agentActionEstimateProj})  occurs when the agents are at a consensus and at NE.
\begin{lemma} \label{lemma:eqIsConsensusProj}
	Consider  a game $\mathcal{G}(\mathcal{I},J_i,\Omega_i)$ over a communication graph $G_c$ under Assumptions \ref{asmp:CxnGraph}, \ref{asmp:Jsmooth}(ii) and \ref{asmp:strongPseudo}(i) or (ii).  Let the dynamics for each agent  $\widetilde{\mathcal{P}}_i$ be as in  (\ref{eq:agentActionEstimateProj}), or overall $\widetilde{\mathcal{P}}$,  (\ref{eq:overallDynProj}). At an equilibrium point $\overline{\mathbf{x}}$ the estimate vectors of all players are equal $\bar{\mathbf{x}}^i=\bar{\mathbf{x}}^j$, $\forall i,j\in\mathcal{I}$ and equal to the Nash equilibrium profile $x^*$, hence the action components of all players  coincide with the optimal actions, $\bar{\mathbf{x}}^i_i = x^*_i$, $\forall i\in\mathcal{I}$.
\end{lemma}
\begin{proof}
Let $\overline{\mathbf{x}}$ denote an equilibrium of (\ref{eq:overallDynProj}),  
\begin{align}
\label{eq:LeqPointProj}
\mathbf{0}_{Nn} &= \mathcal{R}^T\Pi_{\Omega}\left (\mathcal{R}\overline{\mathbf{x}},-\mathbf{F}(\bar{\mathbf{x}})-\mathcal{R}\mathbf{L}\bar{\mathbf{x}}\right ) - \mathcal{S}^T\mathcal{S}\mathbf{L}\bar{\mathbf{x}}
\end{align}
Pre-multiplying both sides by $\mathcal{R}$ and using (\ref{eq:IdRSMatrix}) simplifies to,
\begin{align}
\mathbf{0}_{n} = \Pi_{\Omega}\left (\mathcal{R}\overline{\mathbf{x}},-\mathbf{F}(\bar{\mathbf{x}})-\mathcal{R}\mathbf{L}\bar{\mathbf{x}} \right ) \label{eq:EqExPsGradProj}
\end{align}
Substituting (\ref{eq:EqExPsGradProj}) into (\ref{eq:LeqPointProj}) results in $\mathbf{0}_{Nn} = - \mathcal{S}^T\mathcal{S}\mathbf{L}\bar{\mathbf{x}}$ which implies $ \bar{\mathbf{x}}\in null(\mathbf{L})$. 
From this it follows that $\bar{\mathbf{x}}^i = \bar{\mathbf{x}}^j$, $\forall i,j\in\mathcal{I}$ by Assumption \ref{asmp:CxnGraph} and  (\ref{eq:nullExpL}). Therefore   $\bar{\mathbf{x}} = \mathbf{1}_N\otimes \bar{x}$, for some $ \bar{x} \in {\Omega}$. 
Substituting this back into (\ref{eq:EqExPsGradProj}) yields 
$\mathbf{0}_{n} = \Pi_{\Omega}\left (\mathcal{R}(\mathbf{1}_N\otimes \bar{x}),-\mathbf{F}(\mathbf{1}_N\otimes \bar{x}) \right)$
or $\mathbf{0}_{n} = \Pi_{\Omega}(\bar{x},-F(\bar{x}))$ by using (\ref{eq:expPsuedoGrad}). 
Therefore as in (\ref{eq:NE_proj_with_n}) it follows that  $-F(\bar{x})\in N_{\Omega}(\bar{x})$ hence, by (\ref{eq:ViNashNormCone}), $\bar{x} =x^*$ the NE. Thus  $\bar{\mathbf{x}} = \mathbf{1}_N\otimes x^*$ and for all $i,j \in\mathcal{I}$, $\bar{\mathbf{x}}^i = \bar{\mathbf{x}}^j=x^*$ the NE of the game. 
\end{proof}

The following results show  single-timescale convergence to the NE of the game over \emph{a connected} $G_c$,  under Assumption \ref{asmp:strongPseudo}(i) or Assumption \ref{asmp:strongPseudo}(ii). 

\begin{thm} \label{thm:strongPseudoCommProj}
Consider a game $\mathcal{G}(\mathcal{I},J_i, \Omega_i)$ over a communication graph $G_c$ under Assumptions \ref{asmp:CxnGraph}, \ref{asmp:Jsmooth}(ii), \ref{asmp:PseudoGradMono}(i) and \ref{asmp:strongPseudo}(i).  Let each player's dynamics $\widetilde{\mathcal{P}}_i$, be as in (\ref{eq:agentActionEstimateProj}), or overall $\widetilde{\mathcal{P}}$,   (\ref{eq:overallDynProj}).  Then, for any $x(0) \in \Omega$ and any $z(0)$, the solution of (\ref{eq:overallDynProj}) asymptotically converges to $\mathbf{1}_N\otimes x^*$, and the actions components converge to the NE of the game, $x^*$.
\end{thm}
\begin{proof}
The proof  is similar to the proof of Theorem \ref{thm:strongPseudoComm} except that,  instead of LaSalle's invariance principle, the argument is based on Barbalat's Lemma, \cite{K02}, since the system is time-varying. Let $V(t,\mathbf{x}) = \frac{1}{2}\|\mathbf{x}(t)-\overline{\mathbf{x}}\|^2$, where by Lemma \ref{lemma:eqIsConsensusProj},  $\overline{\mathbf{x}} = \mathbf{1}_N\otimes x^*$. Using (\ref{eq:EIP_projbis}) in Lemma \ref {lemma:ForwardEIPProj}, for $\mathbf{x'}(t)=\overline{\mathbf{x}}$, $\mathbf{u}(t) = - \mathbf{L} \mathbf{x}(t)$,  $\mathbf{u'}(t) = - \mathbf{L} \overline{\mathbf{x}}$, it follows that for any $x(0) \in \Omega$ and any $z(0)$, along  (\ref{eq:overallDynProj}), 
\begin{align}\label{dotV_Proj_Thms}
	\dot{V} \leq & -(\mathbf{x}(t)-\bar{\mathbf{x}})^T\mathcal{R}^T(\mathbf{F}(\mathbf{x}(t)) - \mathbf{F}(\bar{\mathbf{x}}))\\
	&-(\mathbf{x}(t)-\bar{\mathbf{x}})^T\mathbf{L}(\mathbf{x}(t)-\bar{\mathbf{x}}) \notag
\end{align}
Under Assumption \ref{asmp:strongPseudo}(i), $\dot{V} \leq 0$, for all $\mathcal{R}\mathbf{x}(t) \in \Omega$, $z(0)$. Thus $V(t,\mathbf{x}(t))$ is non-increasing and bounded from below by $0$, hence it converges as $t \rightarrow \infty$ to some $\underline{V} \geq 0$. Then, under Assumption \ref{asmp:strongPseudo}(i), it follows that $\lim_{t\rightarrow \infty} \int_0^t (\mathbf{x}(\tau)-\bar{\mathbf{x}})^T\mathbf{L}(\mathbf{x}(\tau)-\bar{\mathbf{x}}) d \tau $ exists and is finite. Since $\mathbf{x}(t)$ is absolutely continuous, hence uniformly continuous, from Barbalat's Lemma in \cite{K02} it follows that $ \mathbf{L} ( \mathbf{x}(t) - \bar{\mathbf{x}}) \rightarrow 0$ as $t\rightarrow \infty$. Since $\overline{\mathbf{x}} = \mathbf{1}_N\otimes x^*$, this means that $\mathbf{x}(t) \rightarrow \mathbf{1}_N\otimes x$, as $t \rightarrow \infty$  for some $x \in \Omega$. Then $V(t,\mathbf{x}(t)) = \frac{1}{2}\|\mathbf{x}(t)-  \overline{\mathbf{x}} \|^2 \rightarrow 
\frac{1}{2}\|\mathbf{1}_N\otimes (x -x^* )\|^2= \underline{V}$ as $t \rightarrow \infty$. If $\underline{V} =0$ the proof if completed.  Using  the strict monotonicity assumption \ref{asmp:PseudoGradMono}(i), it can be shown  by a contradiction argument that $x=x^*$ and $\underline{V}=0$. Assume that $x\neq x^*$ and $\underline{V}>0$. Then from (\ref{dotV_Proj_Thms}) there exists a sequence $\{ t_k \}, t_k \rightarrow \infty$, as $k \rightarrow \infty$, such that $\dot{V}(t_k) \rightarrow 0$ as $k \rightarrow \infty$. 
Suppose this claim is false. Then there exists a $d >0$ and a $T>0$ such that $\dot{V}(t) \leq - d$, for all $ t > T$, which contradicts $\underline{V}>0$, hence the claim is true. 
Substituting $\{ t_k \} $ into  (\ref{dotV_Proj_Thms}) yields
\begin{align*}
	\dot{V}(t_k)  \leq & -(\mathbf{x}(t_k)-\bar{\mathbf{x}})^T\mathcal{R}^T(\mathbf{F}(\mathbf{x}(t_k)) - \mathbf{F}(\bar{\mathbf{x}}))\\
	&-(\mathbf{x}(t_k)-\bar{\mathbf{x}})^T\mathbf{L}(\mathbf{x}(t_k)-\bar{\mathbf{x}}) \notag
\end{align*}
where the left-hand side converges to $0$ as $k \rightarrow \infty$. Hence, 
\begin{align*}
0  \leq & - \lim_{k\rightarrow \infty} (\mathbf{x}(t_k)-\bar{\mathbf{x}})^T\mathcal{R}^T(\mathbf{F}(\mathbf{x}(t_k)) - \mathbf{F}(\bar{\mathbf{x}}))\\
	&- \lim_{k\rightarrow \infty} (\mathbf{x}(t_k)-\bar{\mathbf{x}})^T\mathbf{L}(\mathbf{x}(t_k)-\bar{\mathbf{x}}) \notag
\end{align*}
Using $\lim_{k \rightarrow \infty} \mathbf{x}(t_k) =  \mathbf{1}_N\otimes x \in Null(\mathbf{L})$, this leads to 
\begin{align*}
0  \leq & -   [\mathbf{1}_N\otimes ( x -x ^*)]^T\mathcal{R}^T(\mathbf{F}(\mathbf{1}_N\otimes x) - \mathbf{F}(\mathbf{1}_N\otimes x^*))
\end{align*}
or,  $0 \leq -   ( x -x ^*)]^T (F(x) - F( x^*)) < 0$, by the strict monotonicity Assumption \ref{asmp:PseudoGradMono}(i), since we assumed $ x \neq x^*$. This  is a contradiction, hence $x=x^*$ and $\underline{V}=0$.  
\end{proof}


\begin{thm} \label{thm:LipPseudoCommProj}
Consider a game $\mathcal{G}(\mathcal{I},J_i, \Omega_i)$ over a communication graph $G_c$ under Assumptions \ref{asmp:CxnGraph}, \ref{asmp:Jsmooth}(ii), \ref{asmp:PseudoGradMono}(ii) and  \ref{asmp:strongPseudo}(ii).  Let each player's dynamics $\widetilde{\mathcal{P}}_i$ be as in (\ref{eq:agentActionEstimateProj}) or overall $\widetilde{\mathcal{P}}$   (\ref{eq:overallDynProj}).  Then, if 
$\lambda_2(L)  >\frac{\theta^2}{\mu} + \theta$, for any  $x(0) \in \Omega$ and any $z(0)$, the solution of (\ref{eq:overallDynProj}) asymptotically converges to $\mathbf{1}_N\otimes x^*$, and the actions converge to the NE of the game, $x^*$. If $\lambda_2(L)  >\frac{N \theta^2}{\mu} + \theta$, then convergence is exponential.  
\end{thm}

\begin{proof} The proof is similar to the proof of Theorem \ref{thm:LipPseudoComm}. Based on Lemma \ref{lemma:eqIsConsensusProj}, using $V(t,\mathbf{x},\overline{\mathbf{x}}) = \frac{1}{2}\|\mathbf{x}(t)-\overline{\mathbf{x}}\|^2$ and (\ref{eq:EIP_projbis}) in Lemma \ref {lemma:ForwardEIPProj}, for $\mathbf{u}(t) = - \mathbf{L} \mathbf{x}(t)$,  $\mathbf{u'}(t) = - \mathbf{L} \overline{\mathbf{x}}(t)$ one can obtain  (\ref{dotV_Proj_Thms}) along  (\ref{eq:overallDynProj}), for any $x(0) \in \Omega$ and any $z(0)$. Then further decomposing $\mathbf{x}(t)$ into $\mathbf{x}^\perp(t)$ and $ \mathbf{x}^{||}(t)$ components as in the proof of Theorem \ref{thm:LipPseudoComm} leads to an inequality as (\ref{eq:dotVstrong_asy}), where $\Theta$ is positive under the conditions in the theorem. Then invoking Barbalat's Lemma in \cite{K02} as in the proof of Theorem \ref{thm:strongPseudoCommProj} leads to  $x(t) \rightarrow x^*$ and $\mathbf{x}^\perp(t) \rightarrow 0$ as $t \rightarrow \infty$.
\end{proof}

Note also that we can consider a  gain parameter $\frac{1}{\epsilon}>0$ on the estimates in (\ref{eq:overallDynProj}) to improve the lower bound to $ \lambda_2(L)$. Consider
 \begin{align} \label{eq:overallDynProjEps}
	\widetilde{\mathcal{P}}_{\epsilon} :& \quad 
 		\dot{\mathbf{x}} = \mathcal{R}^T\Pi_{\Omega}(\mathcal{R}\mathbf{x},-\mathbf{F}(\mathbf{x})- \frac{1}{\epsilon}\mathcal{R}\mathbf{L}\mathbf{x}) - \frac{1}{\epsilon}\mathcal{S}^T\mathcal{S}\mathbf{L}\mathbf{x}
\end{align}
Thus player $i$'s dynamics is as follows:
\begin{align} \label{eq:agentActionEstimateEPSProj1}
	\widetilde{\mathcal{P}}_{i,\epsilon} : \begin{cases}
	\begin{bmatrix}
	\dot{x}_i \\
	\\
	\dot{\mathbf{x}}^i_{-i}
	\end{bmatrix} =
	\begin{bmatrix}
	\Pi_{\Omega_i}(x_i,-\nabla_i J_i(x_i,\mathbf{x}^i_{-i}) \\
	\quad\quad\quad -\frac{1}{\epsilon}\mathcal{R}_i\sum_{j\in\mathcal{N}_i} ( \mathbf{x}^i - \mathbf{x}^j)) \\
	- \frac{1}{\epsilon}\mathcal{S}_i(\sum_{j\in\mathcal{N}_i} \mathbf{x}^i - \mathbf{x}^j)\\
	\end{bmatrix}
	\end{cases}
\end{align}
Following the proof of Theorem \ref{thm:LipPseudoComm} the matrix $\Theta$ becomes
\begin{align*}
	\Theta = \left [ \begin{array}{cc}  \mu & -\theta \\ -\theta & \frac{1}{\epsilon}\lambda_2(L) - \theta \end{array} \right ]
\end{align*}
where the condition for the matrix to be positive definite is $\lambda_2(L)  > \epsilon(\frac{\theta^2}{\mu} + \theta)$. As in the two-time scale analysis this $\epsilon$ is a global parameter.

\section{Numerical Examples}\label{Simulation}
\subsection{Unconstrained $\Omega$ and Dynamics}

\emph{Example 1:}  Consider a N-player quadratic game from economics, where 20 firms are involved in the production of a homogeneous commodity. The quantity produced by firm $i$ is denoted by $x_i$. The overall cost function of firm $i$  is  $J_i(x_i,x_{-i}) = c_i(x_i) - x_if(x)$, where $c_i(x_i) = (20 + 10(i-1))x_i$ is the production cost, $f(x) = 2200 - \sum_{i\in\mathcal{I}} x_i$ is the demand price, as in \cite{WeiShiACC2017}. We investigate the proposed dynamics  (\ref{eq:agentActionEstimate}) over a communication graph $G_c$ via simulation.   The initial conditions are selected  randomly from $[0,20]$. 
 Assumption \ref{asmp:PseudoGradMono}(i) and \ref{asmp:strongPseudo}(i) hold, so by Theorem \ref{thm:strongPseudoComm}  the dynamics  (\ref{eq:agentActionEstimate}) will converge even over a minimally connected graph. Figures  \ref{2Random} and \ref{2aCom} show the  convergence of (\ref{eq:agentActionEstimate})   over  a randomly generated communication graph $G_c$  (Fig. \ref{2_random_Graph}) and over a cycle $G_c$ graph (Fig. \ref{2cycleComm}), respectively.

\begin{figure}[ht]
\centering
\begin{minipage}[ht]{0.45\columnwidth}
\vspace{-1.8cm}	\centerline{\includegraphics[width=5cm]{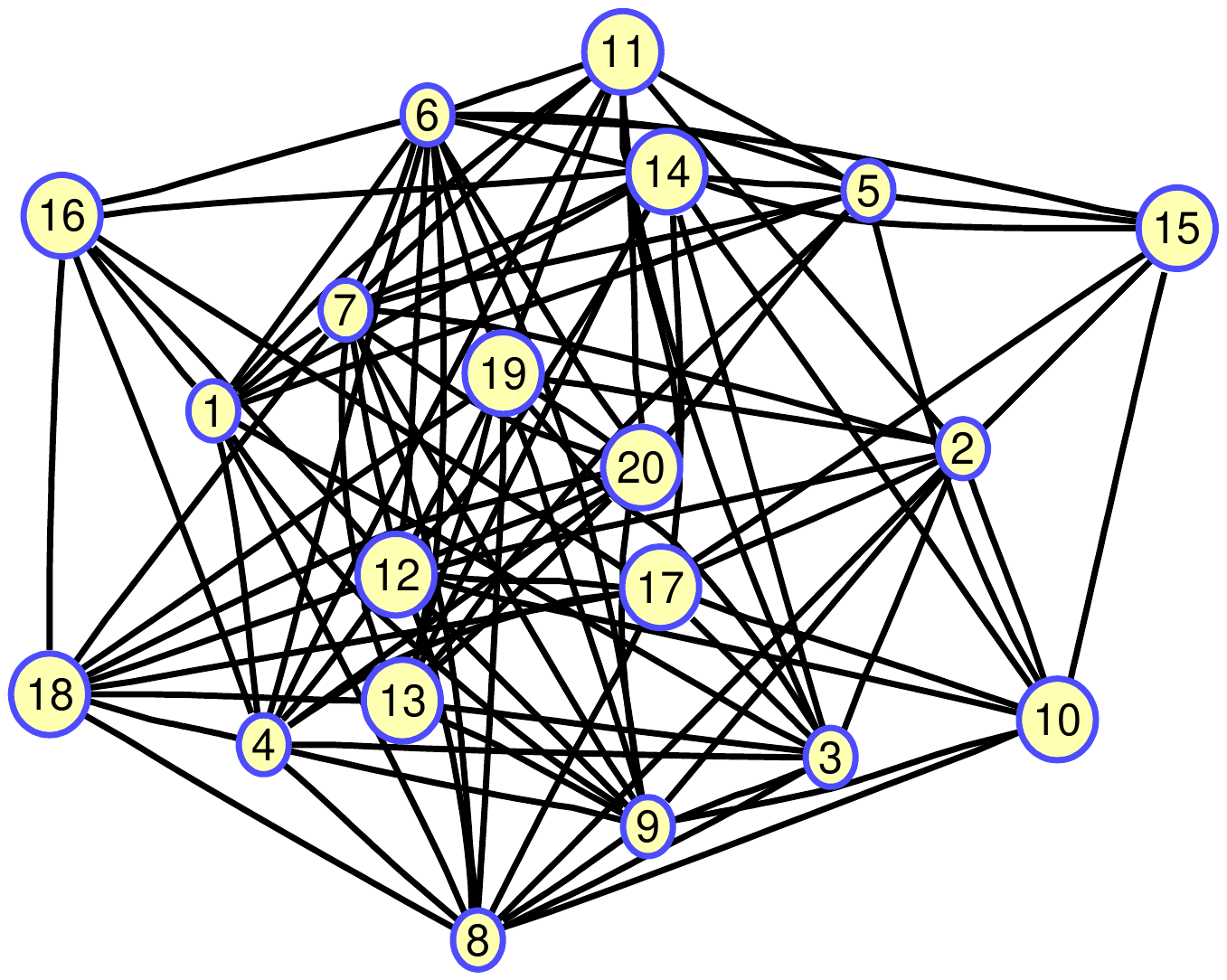}}
\vspace{-2.2cm}	\caption{Random $G_c$, $\lambda_2=4.95$}\label{2_random_Graph}
\end{minipage}
\begin{minipage}[ht]{0.45\columnwidth}
\[
\xymatrixrowsep{4mm}
\xymatrixcolsep{4mm}
    \xymatrix{
     & *+[o][F]{1}\ar@{-}[r] & *+[o][F]{2}\ar@{-}[rd] & \\
    *+[o][F]{20}\ar@{-}[ru] & & & *+[o][F]{3} \\
     & *+[o][F]{5}\ar@{.}[lu]\ar@{-}[r] & *+[o][F]{4}\ar@{-}[ru] & \\
     }
\]
	\caption{Cycle $G_c$ Graph}
	\label{2cycleComm}
\end{minipage}
\end{figure}

\begin{figure}[ht]
\centering
\begin{minipage}[ht]{0.45\columnwidth}
	\centerline{\includegraphics[width=4cm]{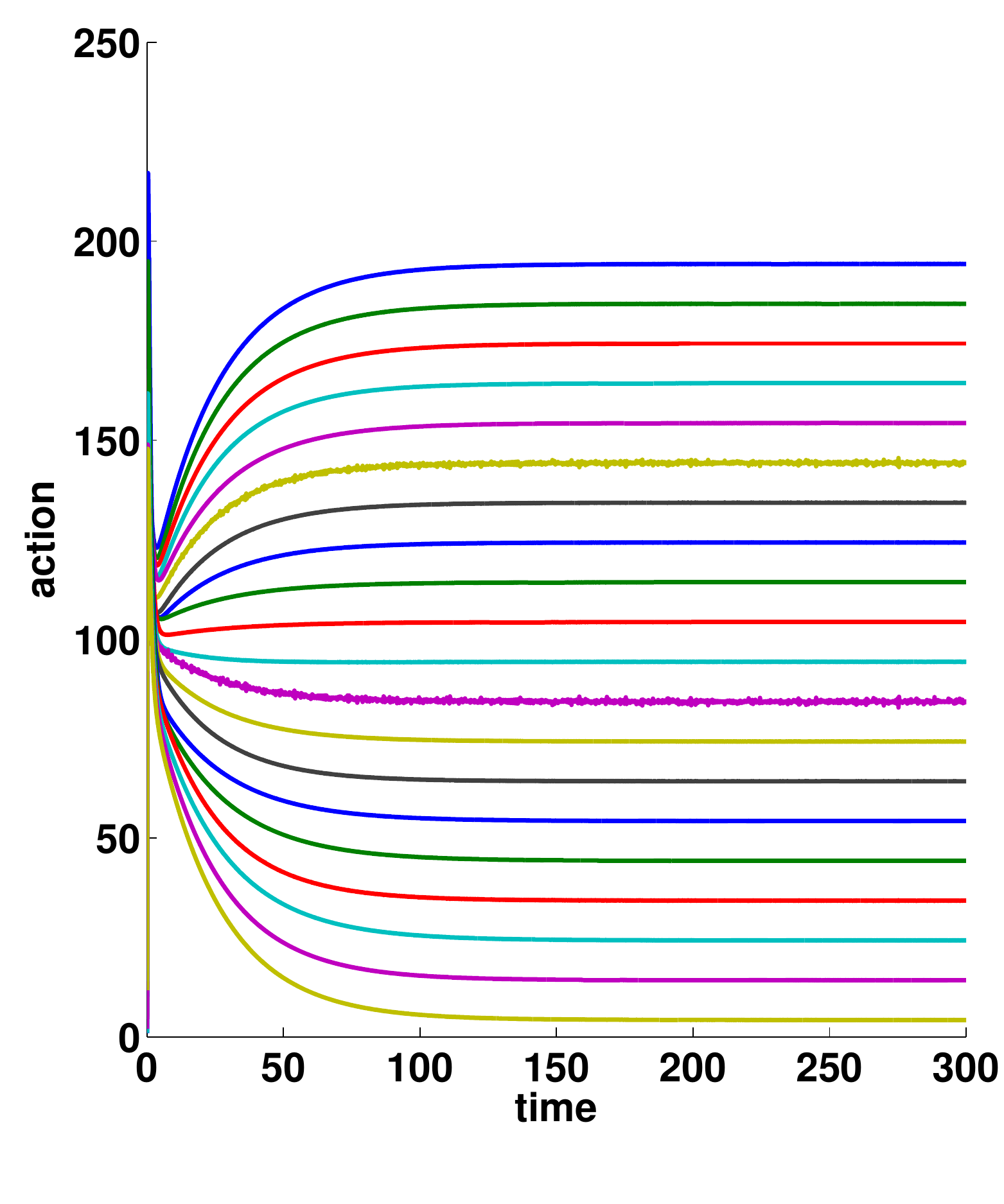}}
	\caption{(\ref{eq:agentActionEstimate}) over random $G_c$}\label{2Random}
	\end{minipage}
\begin{minipage}[ht]{0.45\columnwidth}
	\centerline{\includegraphics[width=4cm]{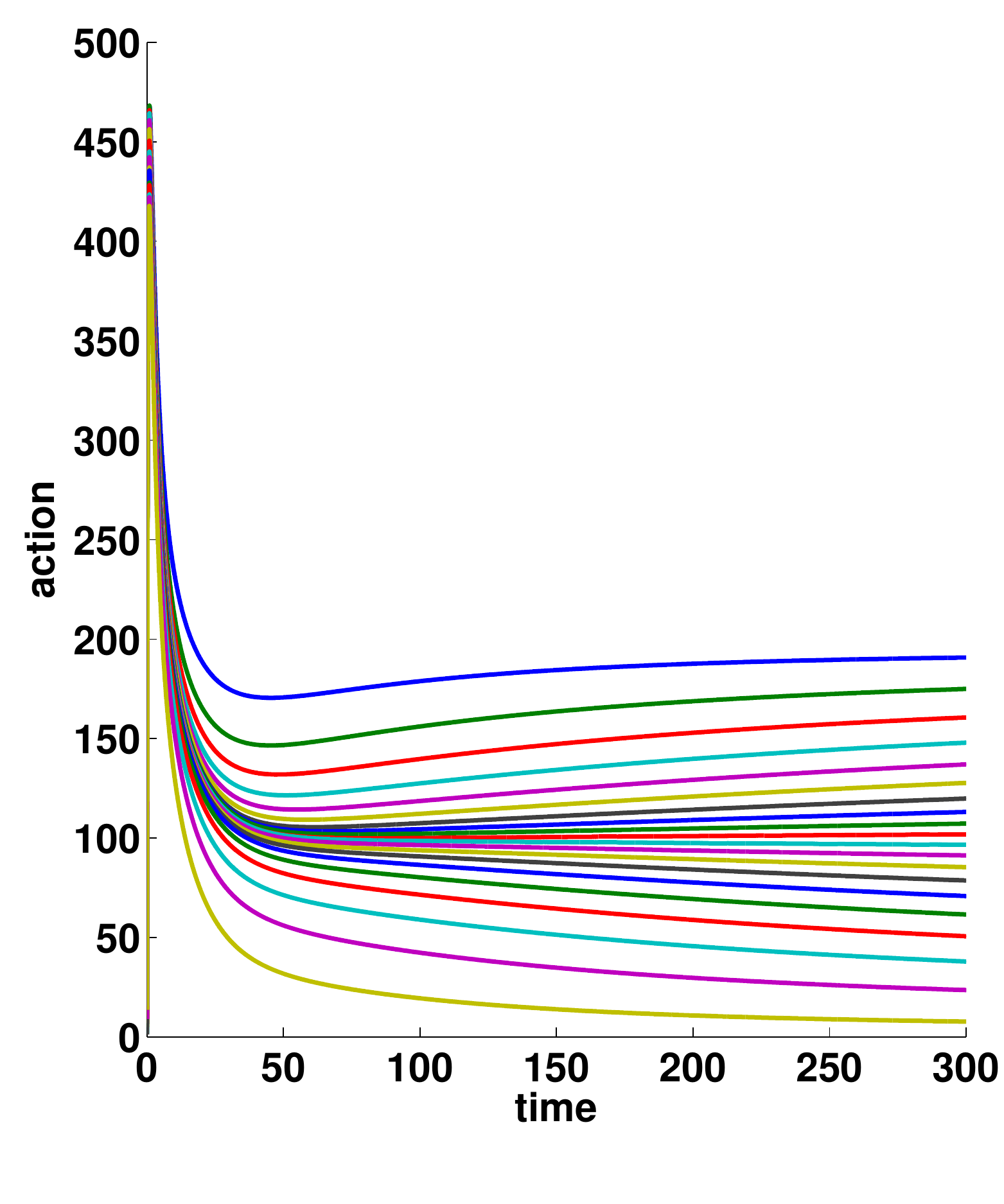}}
	\caption{(\ref{eq:agentActionEstimate}) over cycle $G_c$}\label{2aCom}
\end{minipage}
\end{figure}

%
%


\emph{Example 2:}  
Consider a second example of an 8 player  game with   $J_i(x_i,x_{-i}) = c_i(x_i) - x_if(x)$, $c_i(x_i) = (10 + 4(i-1))x_i$, $f(x) = 600 - \sum_{i\in\mathcal{I}} x_i^2$, as in  \cite{FPAuto2016}. 
Here Assumption \ref{asmp:strongPseudo}(i) on $\mathbf{F}$ does not hold globally, so cannot apply Theorem \ref{thm:strongPseudoComm}, but Assumption \ref{asmp:strongPseudo}(ii) holds locally. 
 By Theorem  \ref{thm:LipPseudoComm},  (\ref{eq:agentActionEstimate})   will converge depending on $\lambda_2(L)$. Figure  \ref{1Random} shows the  convergence of (\ref{eq:agentActionEstimate})  over  a sufficiently connected, randomly generated communication graph $G_c$  as depicted in Fig. \ref{1_random_Graph}. Over a cycle $G_c$ graph, (\ref{eq:agentActionEstimate}) does not converge. Alternatively, by Theorem  \ref{thm:STRGMonPseudoComm}, a higher $1/\epsilon$  (time-scale decomposition) can balance the connectivity loss.
Fig. \ref{1bComCycle} shows convergence  for  (\ref{eq:agentActionEstimateEPS}) with $1/\epsilon = 200$, over a cycle $G_c$ graph  as shown in Fig. \ref{1cycleComm}. 
 The initial conditions are selected  randomly from $[0,20]$. 


\begin{figure}[ht]
\centering
\begin{minipage}[ht]{0.45\columnwidth}
\vspace{-1.8cm}	\centerline{\includegraphics[width=5cm]{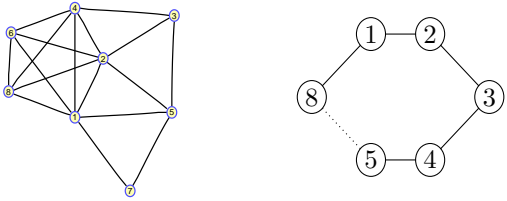}}
\vspace{-2.2cm}	\caption{Random $G_c$, $\lambda_2=1.67$}\label{1_random_Graph}
\end{minipage}
\begin{minipage}[ht]{0.45\columnwidth}
\[
\xymatrixrowsep{4mm}
\xymatrixcolsep{4mm}
    \xymatrix{
     & *+[o][F]{1}\ar@{-}[r] & *+[o][F]{2}\ar@{-}[rd] & \\
    *+[o][F]{8}\ar@{-}[ru] & & & *+[o][F]{3} \\
     & *+[o][F]{5}\ar@{.}[lu]\ar@{-}[r] & *+[o][F]{4}\ar@{-}[ru] & \\
     }
\]
	\caption{Cycle $G_c$ Graph}
	\label{1cycleComm}
\end{minipage}
\end{figure}

\begin{figure}[ht]
\centering
\begin{minipage}[ht]{0.45\columnwidth}
	\centerline{\includegraphics[width=4cm]{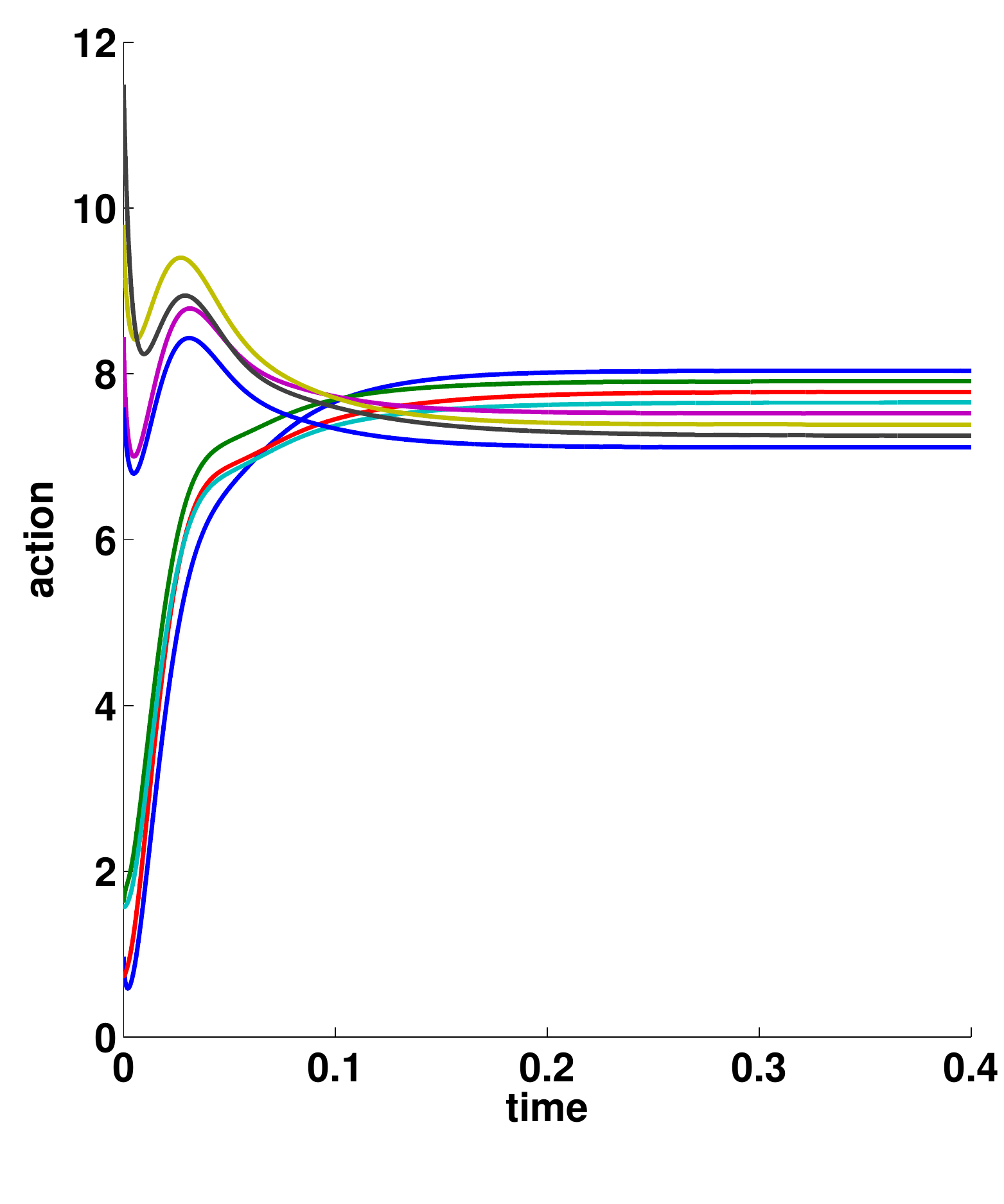}}
	\caption{(\ref{eq:agentActionEstimate}) over random $G_c$}\label{1Random}
\end{minipage}
\begin{minipage}[ht]{0.45\columnwidth}
	\centerline{\includegraphics[width=4cm]{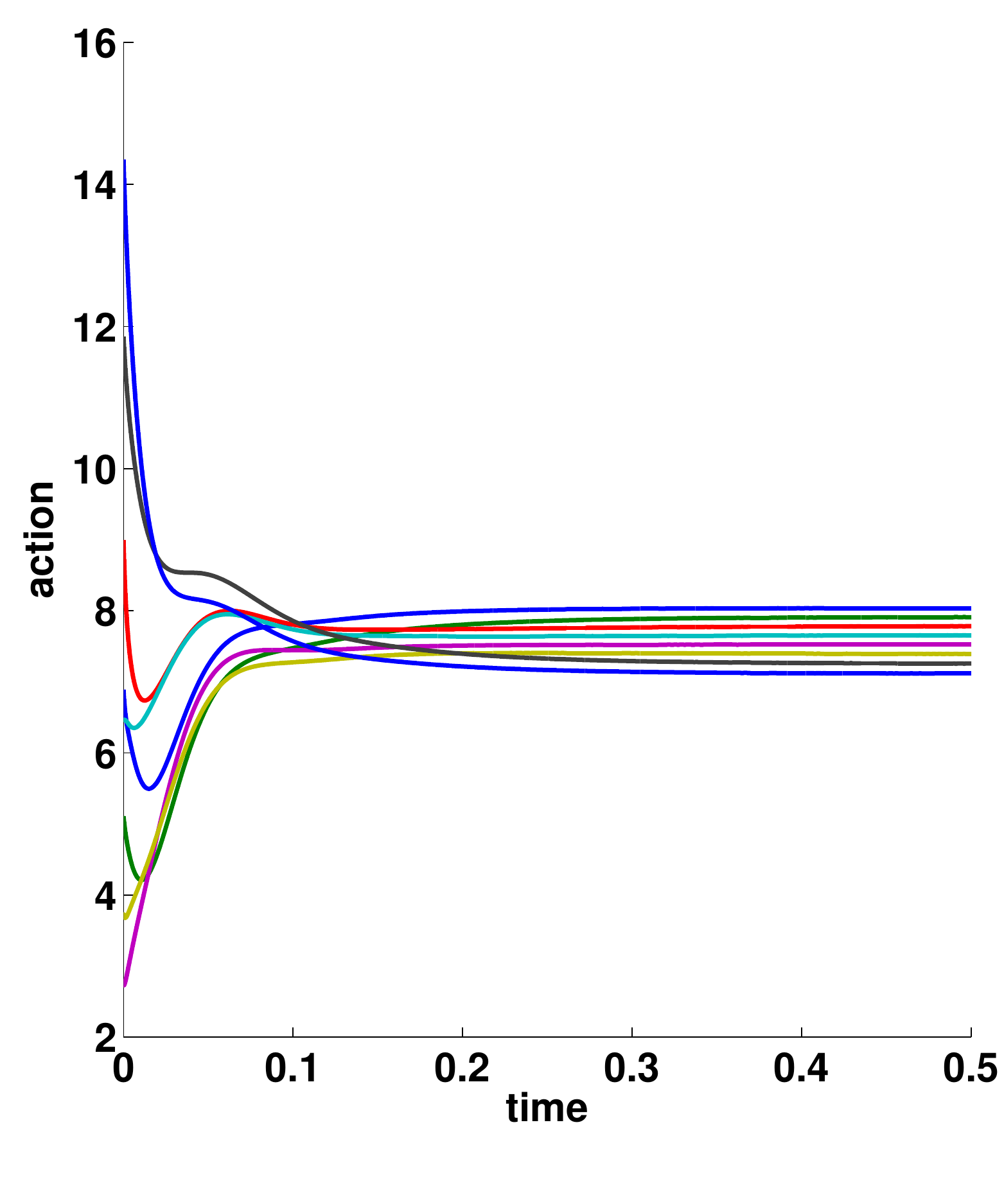}}
	\caption{(\ref{eq:agentActionEstimateEPS}) cycle $G_c$, $1/\epsilon = 200$}\label{1bComCycle}
\end{minipage}
\end{figure}

\subsection{Compact $\Omega$ and Projected Dynamics}


\emph{Example 1:}  $N=20$ and this time $\Omega_i =[0,200]$. We investigate the projected augmented gradient dynamics  (\ref{eq:agentActionEstimateProj}) over a  graph $G_c$.   The actions' initial conditions are selected  randomly from $[0,200]$, while the estimates from $[-2000,2000]$. 
 Assumption \ref{asmp:PseudoGradMono}(i) and \ref{asmp:strongPseudo}(i) hold, so by Theorem \ref{thm:strongPseudoCommProj}  the dynamics  (\ref{eq:agentActionEstimateProj}) will converge even over a minimally connected graph. Figures  \ref{2RandomProj} and \ref{2aComProj} show the  convergence of (\ref{eq:agentActionEstimateProj})   over  a randomly generated communication graph $G_c$  (Fig. \ref{2_random_GraphProj}) and over a cycle $G_c$ graph (Fig. \ref{2cycleCommProj}), respectively.

\begin{figure}[ht]
\centering
\begin{minipage}[ht]{0.45\columnwidth}
\vspace{-1.8cm}	\centerline{\includegraphics[width=5cm]{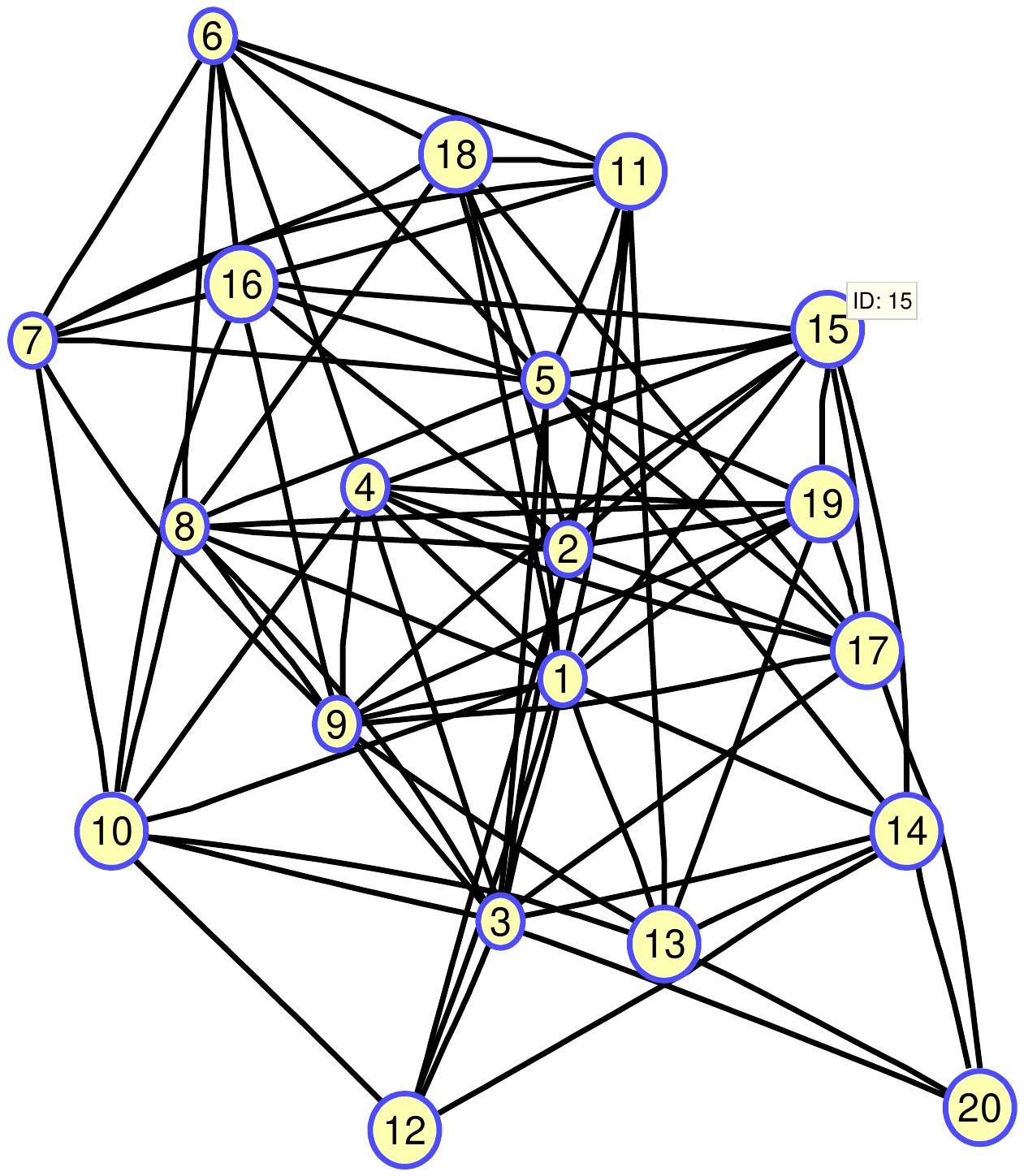}}
\vspace{-1.8cm}	\caption{Random $G_c$, $\lambda_2=3.13$}\label{2_random_GraphProj}
\end{minipage}
\begin{minipage}[ht]{0.45\columnwidth}
\[
\xymatrixrowsep{4mm}
\xymatrixcolsep{4mm}
    \xymatrix{
     & *+[o][F]{1}\ar@{-}[r] & *+[o][F]{2}\ar@{-}[rd] & \\
    *+[o][F]{20}\ar@{-}[ru] & & & *+[o][F]{3} \\
     & *+[o][F]{5}\ar@{.}[lu]\ar@{-}[r] & *+[o][F]{4}\ar@{-}[ru] & \\
     }
\]
\vspace{0.2cm}	\caption{Cycle $G_c$ Graph}
	\label{2cycleCommProj}
\end{minipage}
\end{figure}

\begin{figure}[ht]
\centering
\begin{minipage}[ht]{0.45\columnwidth}
	\centerline{\includegraphics[width=4cm]{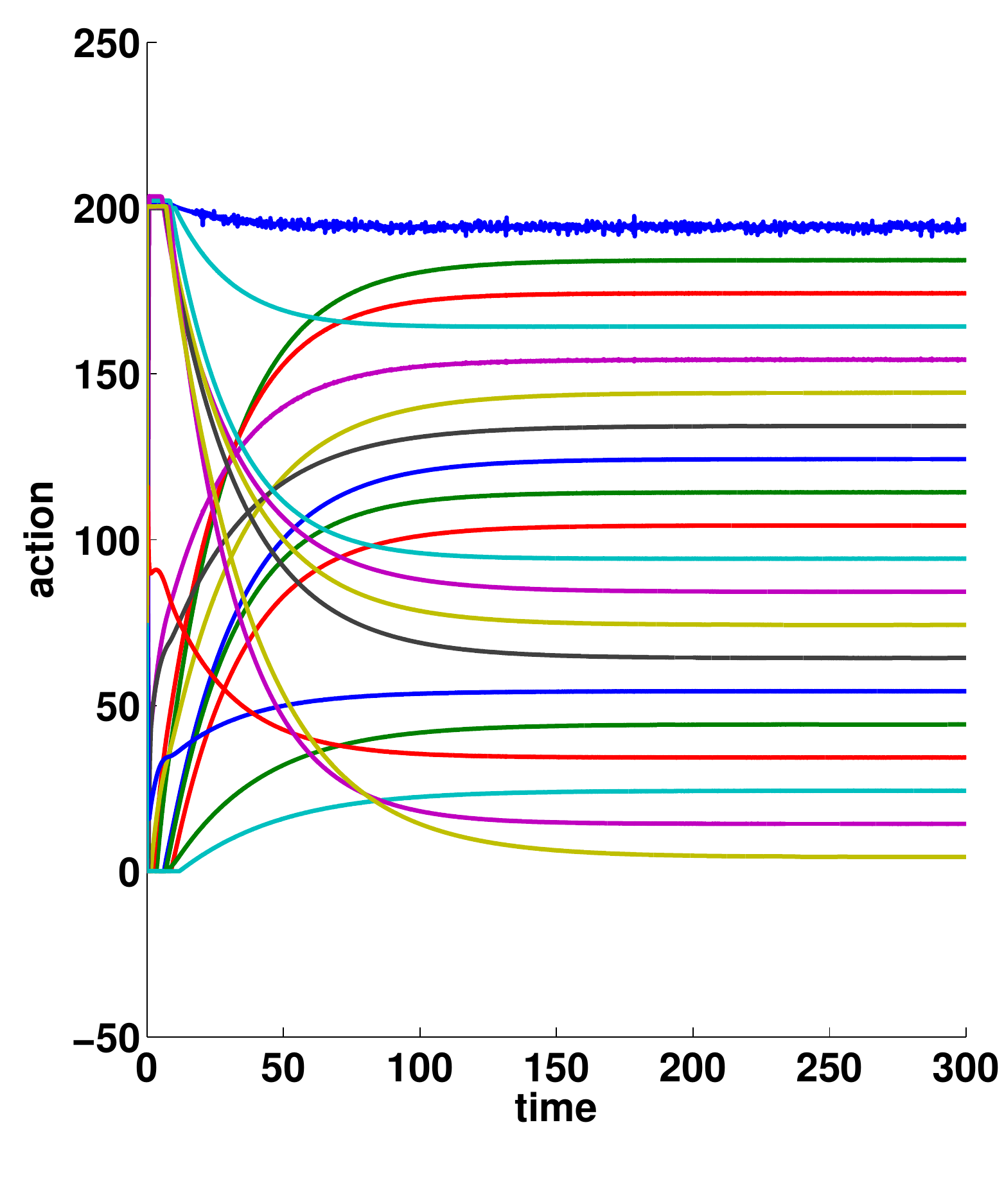}}
	\caption{(\ref{eq:agentActionEstimateProj}) over random $G_c$}\label{2RandomProj}
	\end{minipage}
\begin{minipage}[ht]{0.45\columnwidth}
	\centerline{\includegraphics[width=4cm]{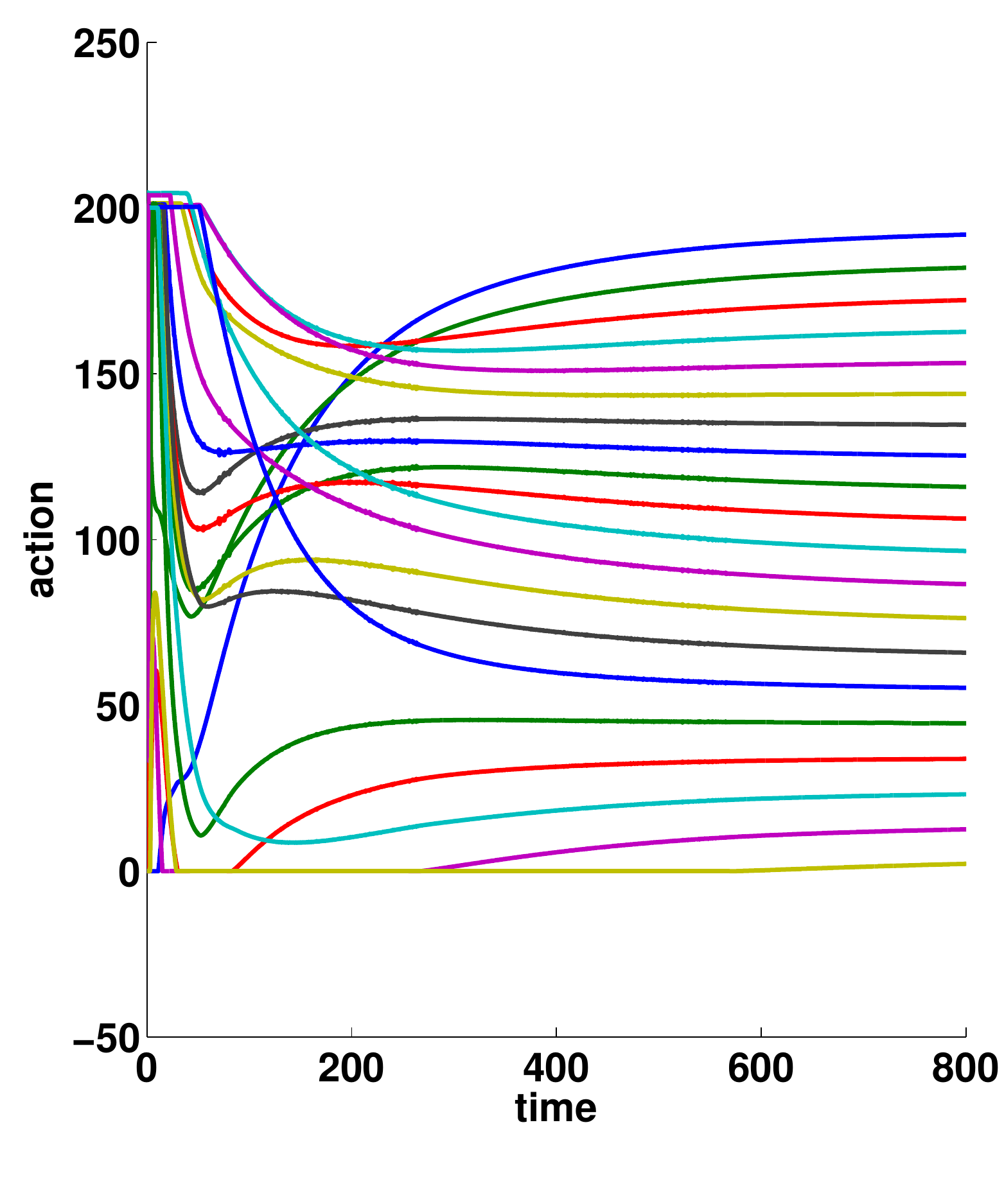}} 
	\caption{(\ref{eq:agentActionEstimateProj}) over cycle $G_c$}\label{2aComProj}
\end{minipage}
\end{figure}

\emph{Example 2: }  $N=8$ and  $\Omega_i =[0,20]$, as in  \cite{FPAuto2016}. Here Assumption \ref{asmp:strongPseudo}(i) on $\mathbf{F}$ does not hold globally, so cannot apply Theorem \ref{thm:strongPseudoCommProj}.  Under Assumption \ref{asmp:PseudoGradMono}(ii) and  \ref{asmp:strongPseudo}(ii) by Theorem  \ref{thm:LipPseudoComm},  (\ref{eq:agentActionEstimateProj})   will converge depending on $\lambda_2(L)$. Figure  \ref{1RandomProj} shows the  convergence of (\ref{eq:agentActionEstimateProj})  over  a sufficiently connected, randomly generated communication graph $G_c$  as in Fig. \ref{1_random_GraphProj}. A  higher $1/\epsilon$ on the estimates can balance the lack of  connectivity. 
Fig. \ref{1bComCycleProj} shows results for  (\ref{eq:agentActionEstimateEPSProj1}) with $1/\epsilon = 200$, over a cycle $G_c$ graph as  in Fig. \ref{1cycleCommProj}.
 The actions' initial conditions are selected  randomly from $[0,20]$ and the estimates from $[-200,200]$.


\begin{figure}[ht]
\centering
\begin{minipage}[ht]{0.45\columnwidth}
\vspace{-1.8cm}	\centerline{\includegraphics[width=5cm]{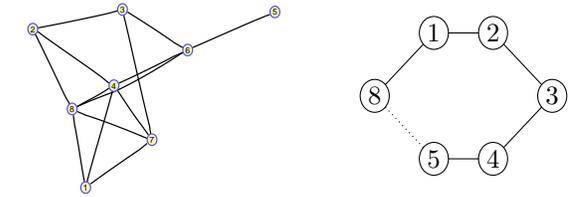}}
\vspace{-2.2cm}	\caption{Random $G_c$, $\lambda_2=0.83$}\label{1_random_GraphProj}
\end{minipage}
\begin{minipage}[ht]{0.45\columnwidth}
\[
\xymatrixrowsep{4mm}
\xymatrixcolsep{4mm}
    \xymatrix{
     & *+[o][F]{1}\ar@{-}[r] & *+[o][F]{2}\ar@{-}[rd] & \\
    *+[o][F]{8}\ar@{-}[ru] & & & *+[o][F]{3} \\
     & *+[o][F]{5}\ar@{.}[lu]\ar@{-}[r] & *+[o][F]{4}\ar@{-}[ru] & \\
     }
\]
	\caption{Cycle $G_c$ Graph}
	\label{1cycleCommProj}
\end{minipage}
\end{figure}

\begin{figure}[ht]
\centering
\begin{minipage}[ht]{0.45\columnwidth}
	\centerline{\includegraphics[width=4cm]{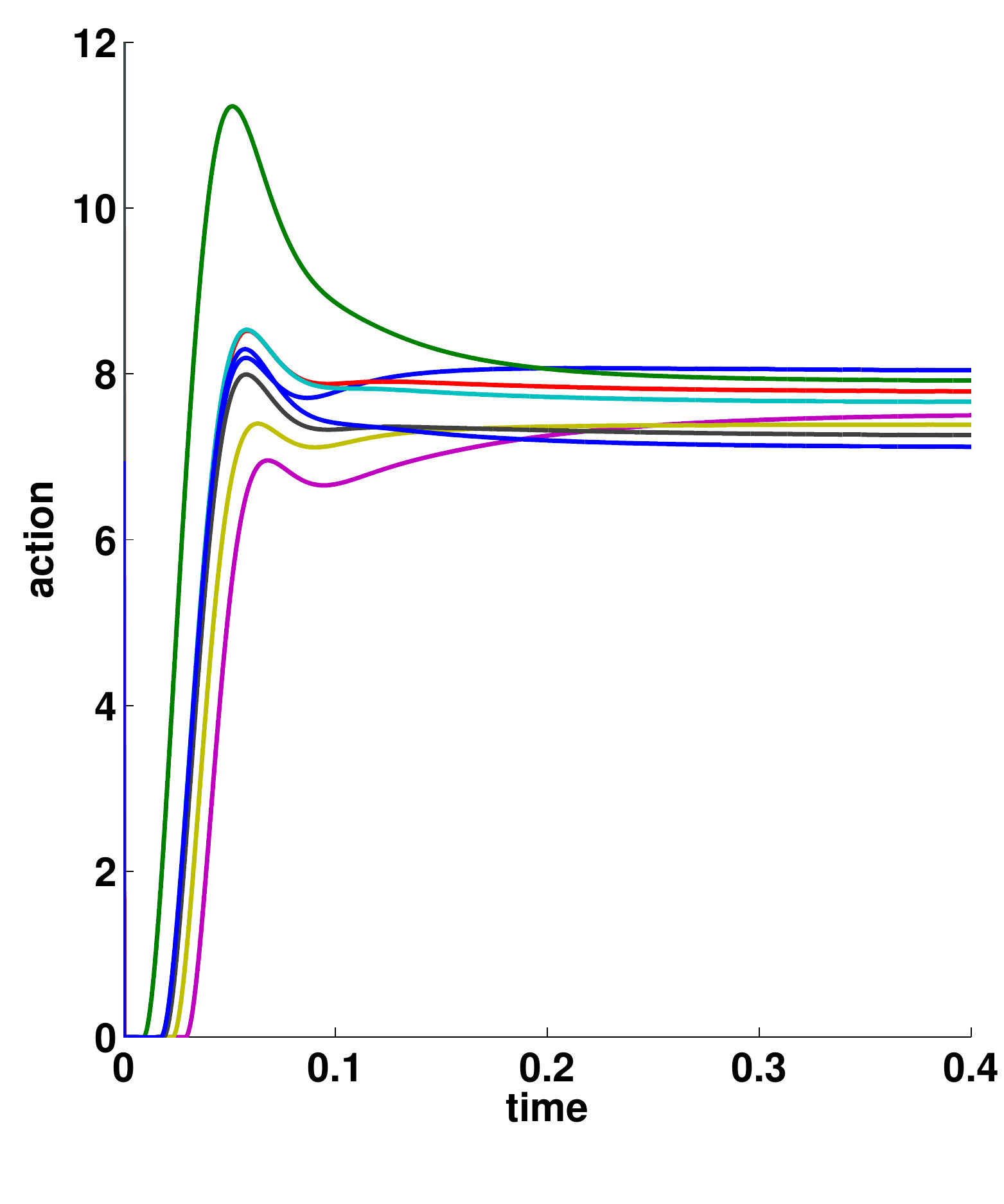}}
	\caption{(\ref{eq:agentActionEstimateProj}) over random $G_c$}\label{1RandomProj}
\end{minipage}
\begin{minipage}[ht]{0.45\columnwidth}
	\centerline{\includegraphics[width=4cm]{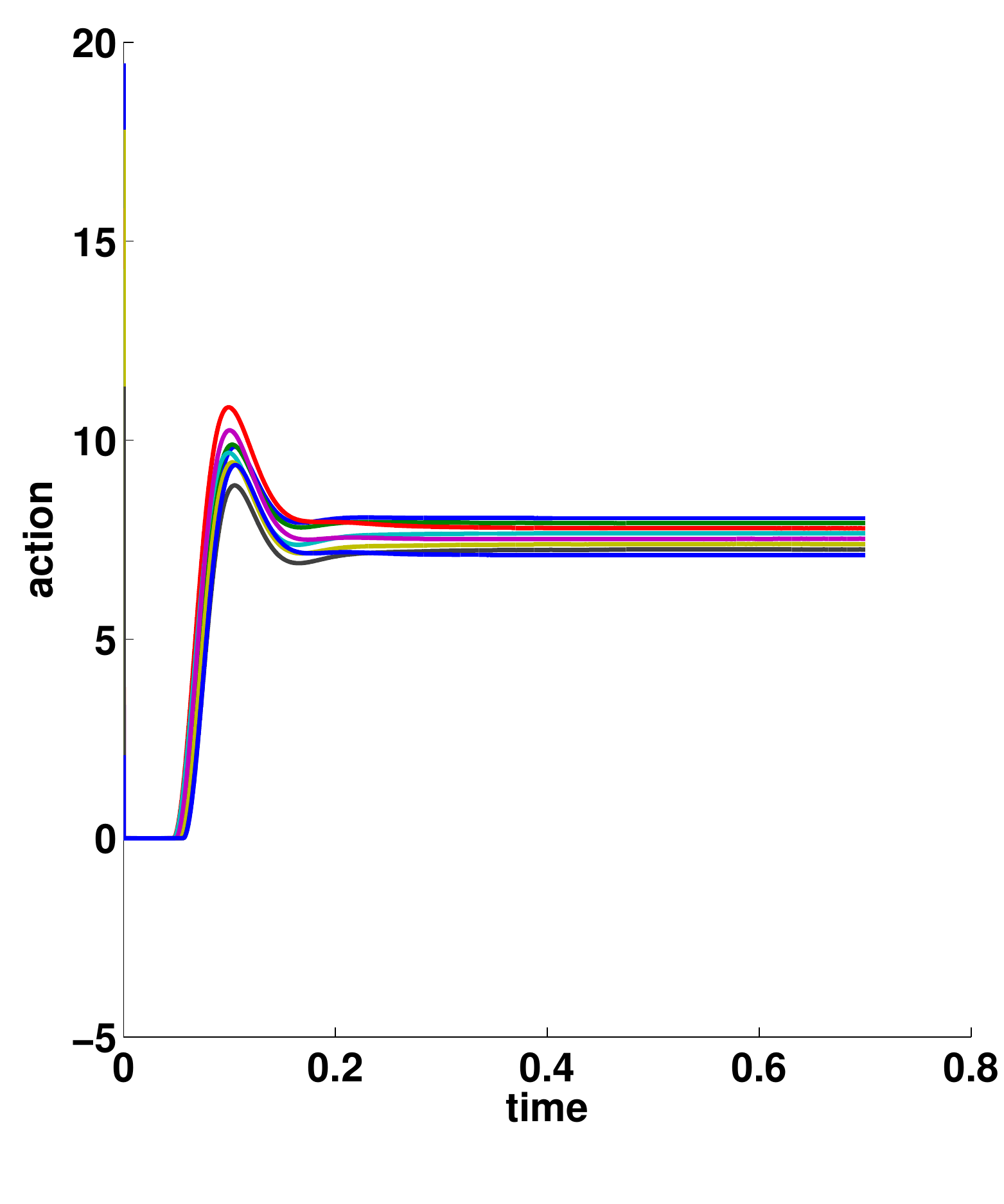}} 
	\caption{(\ref{eq:agentActionEstimateEPSProj1}) cycle $G_c$, $1/\epsilon = 200$}\label{1bComCycleProj}
\end{minipage}
\end{figure}

\emph{Example 3:} Consider    $J_i(x_i,x_{-i}) = c_i(x_i) - x_if(x)$, where $c_i(x_i) = [ 20 + 40(i-1) ] x_i$ and $f(x) = 1200 - \sum_i(x_i)$, $\Omega_i = [0,200]$, for $N=20$. The NE is on the boundary $
\begin{bmatrix} 200 & 200 & 183.3 & 143.3 & 103.3
& 63.3 & 23.3 & 0 & \cdots \end{bmatrix}
$. Figure  \ref{4RandomProj} shows the  convergence of (\ref{eq:agentActionEstimateProj})  over  a randomly generated communication graph $G_c$ as in Fig. \ref{4_random_GraphProj}, while  Fig. \ref{4aComCycleProj} gives similar results, this time over a cycle $G_c$ graph as depicted in Fig. \ref{4cycleCommProj}. 

\begin{figure}[ht]
\centering
\begin{minipage}[ht]{0.45\columnwidth}
\vspace{-1.8cm}	\centerline{\includegraphics[width=5cm]{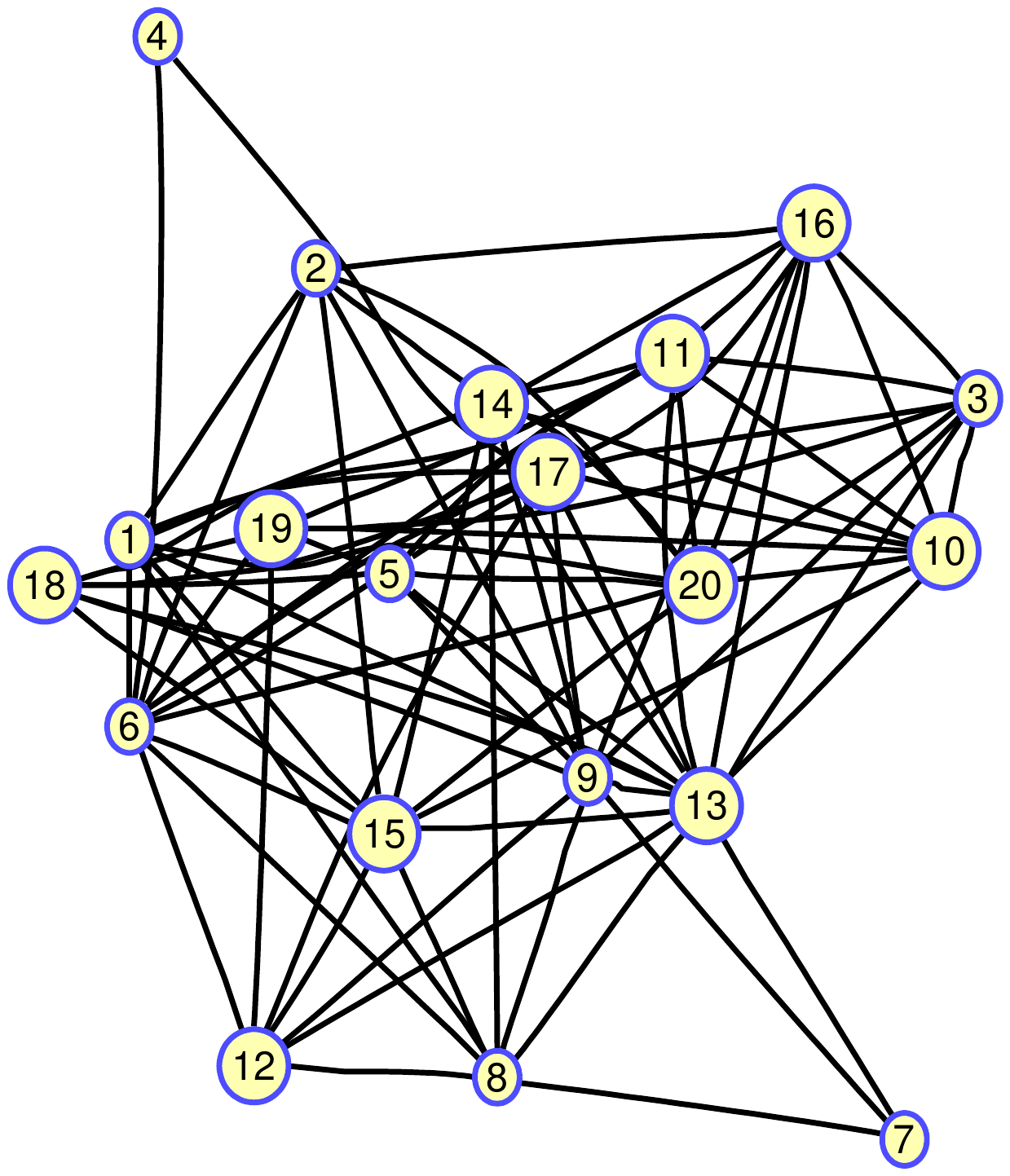}}
\vspace{-1.8cm}	\caption{Random $G_c$, $\lambda_2=1.87$}\label{4_random_GraphProj}
\end{minipage}
\begin{minipage}[ht]{0.45\columnwidth}
\[
\xymatrixrowsep{4mm}
\xymatrixcolsep{4mm}
    \xymatrix{
     & *+[o][F]{1}\ar@{-}[r] & *+[o][F]{2}\ar@{-}[rd] & \\
    *+[o][F]{20}\ar@{-}[ru] & & & *+[o][F]{3} \\
     & *+[o][F]{5}\ar@{.}[lu]\ar@{-}[r] & *+[o][F]{4}\ar@{-}[ru] & \\
     }
\]
\vspace{0.2cm}	\caption{Cycle $G_c$ Graph}
	\label{4cycleCommProj}
\end{minipage}
\end{figure}

\begin{figure}[ht]
\centering
\begin{minipage}[ht]{0.45\columnwidth}
	\centerline{\includegraphics[width=4cm]{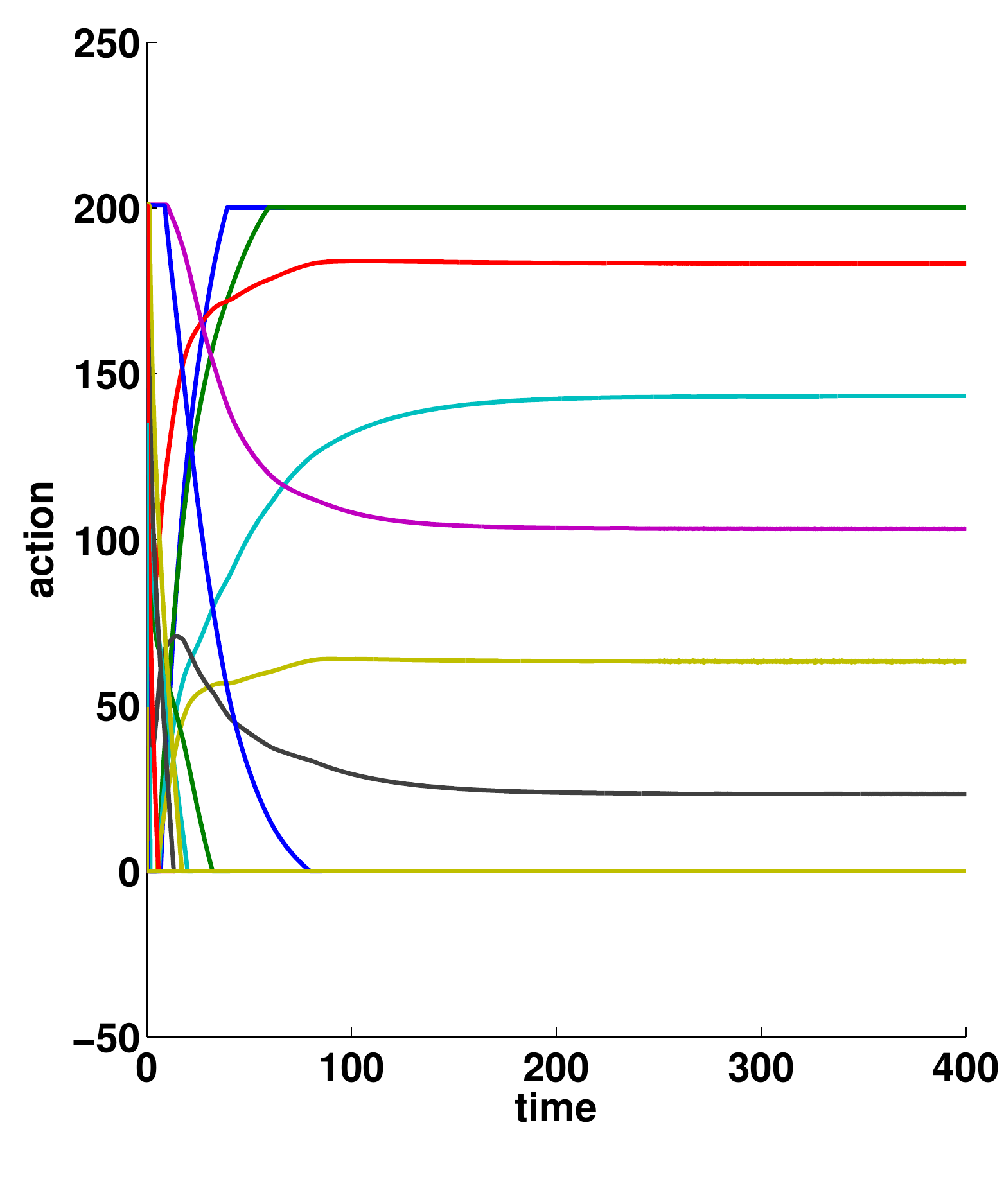}}
	\caption{(\ref{eq:agentActionEstimateProj}) over random $G_c$}\label{4RandomProj}
\end{minipage}
\begin{minipage}[ht]{0.45\columnwidth}
	\centerline{\includegraphics[width=4cm]{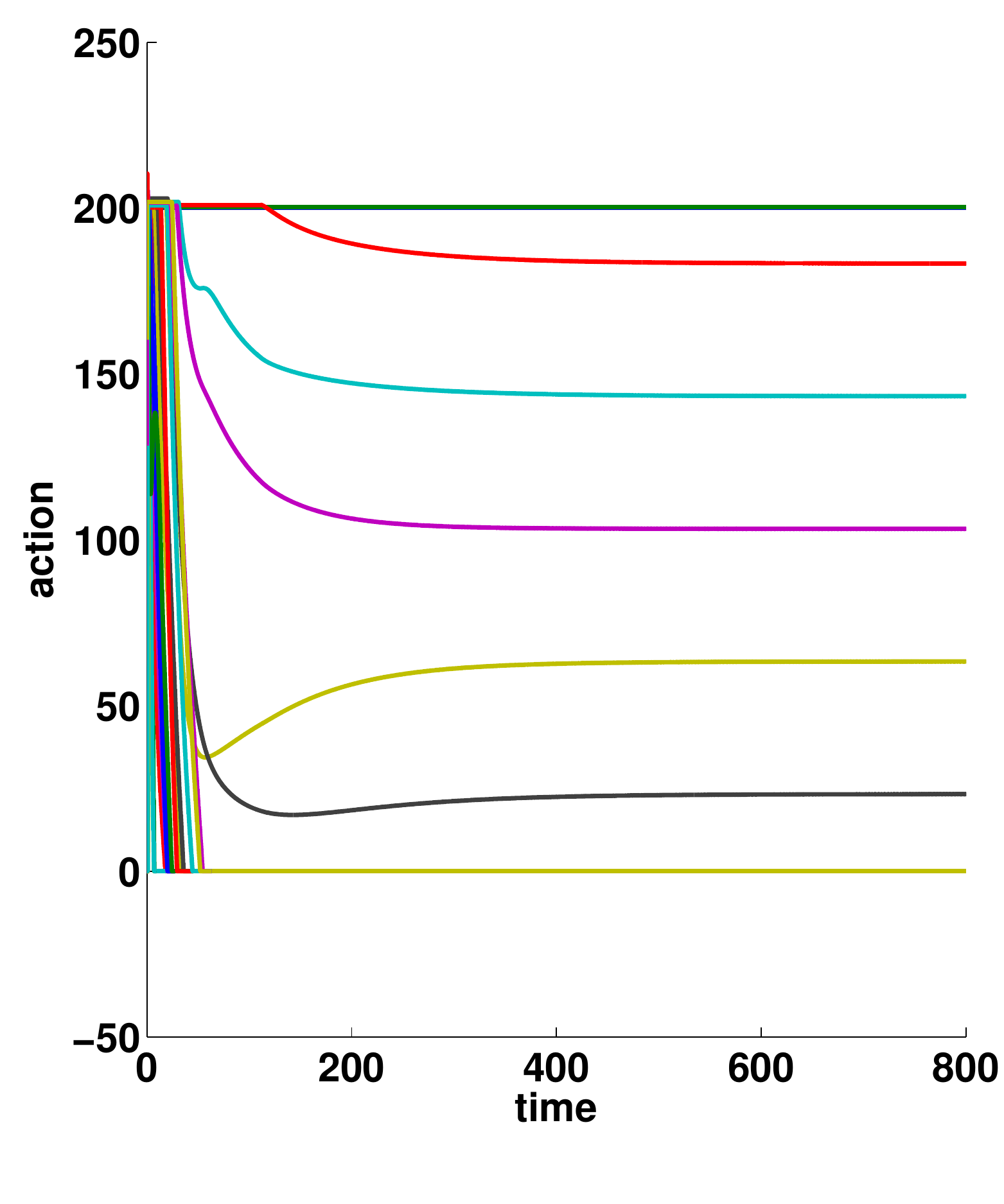}}
	\caption{(\ref{eq:agentActionEstimateProj}) over cycle $G_c$}\label{4aComCycleProj}
\end{minipage}
\end{figure}

\section{Conclusion}\label{Conclusion}

In this paper, we studied distributed Nash equilibrium (NE) seeking over networks in continuous-time. We proposed an augmented   gradient-play dynamics with estimation in which players communicate locally only with their neighbours  to compute an estimate of the other players' actions. We derived the new dynamics based on the reformulation as a multi-agent coordination problem over an undirected graph. We exploited incremental passivity properties and showed that a synchronizing, distributed Laplacian feedback can be designed using relative estimates of the neighbours. Under strict monotonicity  property of the pseudo-gradient, we proved that the new dynamics converges to  the NE of the game. We  discussed  cases that highlight the tradeoff between properties of the game and the communication graph.

\bibliographystyle{IEEEtran}
\bibliography{referencesCDC}


\end{document}